\documentclass[10pt]{amsart}

\usepackage{amssymb}
\usepackage{amsmath}
\usepackage{amsthm}

\usepackage[utf8]{inputenc}
\usepackage{xurl}
\usepackage[T1]{fontenc}
\usepackage{todonotes}
\usepackage{mathtools}
\usepackage{latexsym}
\usepackage{float}
\usepackage{xcolor}
\usepackage{mathrsfs}
\usepackage{calrsfs}
\usepackage[hidelinks]{hyperref}
\usepackage[ruled,linesnumbered]{algorithm2e}
\usepackage{xcolor}
\usepackage{multicol}

\SetKwInput{INPUT}{\hspace*{-2.5mm}Input}
\SetKwInput{OUTPUT}{\hspace*{-2.5mm}Output}

\usepackage{silence}
\numberwithin{equation}{section}
\newtheorem{theo}{Theorem}
\numberwithin{theo}{section}
\newtheorem{coro}[theo]{Corollary}
\newtheorem{prop}[theo]{Proposition}
\newtheorem{defi}[theo]{Definition}
\newtheorem{defiprop}[theo]{Definition-Proposition}
\newtheorem{lemm}[theo]{Lemma}

\newtheorem{expl}[theo]{Example}
\newtheorem{remq}[theo]{Remark}

\newtheorem{prob}[theo]{Problem}

\newcommand{\Z}{\mathbb{Z}}
\newcommand{\N}{\Z_{\geq 0}}
\newcommand{\Q}{\mathbb{Q}}

\newcommand{\F}{\mathbb{F}}

\newcommand{\mO}[1][]{\mathcal{O}_{#1}}

\newcommand{\mOi}{\mO[1]}
\newcommand{\mOd}{\mO[2]}

\newcommand{\B}{\mathcal{B}_{p,\infty}}

\def\commutatif{\ar@{}[rd]^<<<{\circlearrowleft}}

\newcommand{\IsomorphismCompletion}{\textsc{IsomorphismCompletion}}
\newcommand{\LocalGenerator}{\textsc{LocalGenerator}}
\newcommand{\IsomorphismEO}{\textsc{IsomorphismE0}}
\newcommand{\ProductSurfacesIsomorphism}{\textsc{ProductSurfacesIsomorphism}}
\newcommand{\ProductIsomorphism}{\textsc{ProductIsomorphism}}

\DeclareMathOperator{\Conn}{Conn}
\DeclareMathOperator{\Hom}{Hom}
\DeclareMathOperator{\Trd}{Trd}
\DeclareMathOperator{\Nrd}{Nrd}

\DeclareMathOperator{\End}{End}
\DeclareMathOperator{\Aut}{Aut}

\DeclareMathOperator{\val}{val}
\DeclareMathOperator{\rgcd}{rgcd}

\DeclareMathOperator{\Mat}{M}
\DeclareMathOperator{\GL}{GL}

\begin{document}


\title[Computing Isomorphisms between Products of Supersingular Curves]{Computing Isomorphisms between Products of Supersingular Elliptic Curves}

\author{Pierrick Gaudry, Julien Soumier, Pierre-Jean Spaenlehauer
}

\maketitle

%% Abstract
\begin{abstract}
  The Deligne-Ogus-Shioda theorem guarantees the existence of isomorphisms between
  products of supersingular elliptic curves over finite fields.
  In this paper, we present an algorithm for explicitly computing these isomorphisms given
  the endomorphism rings of the curves. Under GRH, it is proved to run in expected
  polynomial time. Our approach leverages the Deuring correspondence,
  enabling us to reformulate computational isogeny problems into algebraic problems in quaternions.
  Specifically, we reduce the computation of isomorphisms to solving systems of quadratic
  and linear equations over the integers derived from norm equations.
  We develop $\ell$-adic techniques for solving these equations when we have access to
  a low discriminant subring. Combining these results leads to the description of an
  efficient probabilistic Las Vegas algorithm for computing the desired isomorphisms.
\end{abstract}

%\begin{keyword}
%Superspecial abelian surfaces \sep Isogenies
%\end{keyword}

% \begin{highlights}
% \end{highlights}
%% Keywords

%% ------------------------------------- Main text

\section{Introduction}

Algorithms for computing isogenies between elliptic curves have been a
vast field of research, leading to the recent development of higher-dimensional techniques
in cryptography~\cite{ComputingIsogeniesTheta, Algo_22_isogenies}.
In particular, abelian varieties of dimension $g \geq 2$
isomorphic to a product of supersingular elliptic curves play an important role
in this setting~\cite{KeyRecoverySIDH_CD, KeyRecoverySIDH_MM, robert2023breaking, 2DWest, PRISM}.
An important feature of such abelian varieties
is that they are all isomorphic over an algebraic closure. 
Let $\mathbb F_q$ be a finite field
of characteristic $p>0$. An abelian variety defined over
$\mathbb F_q$ is \emph{superspecial} if it is $\overline{\mathbb F_q}$-isomorphic to a
product of supersingular elliptic curves defined over $ \overline{\mathbb F_q}$. The
Deligne-Ogus-Shioda theorem~\cite[Thm.~3.5]{DelignTheo} states that for all $g>1$,
all dimension-$g$ superspecial abelian varieties
defined over $\F_q$ are $\overline{\mathbb F_q}$-isomorphic (as unpolarized
abelian varieties). The aim of this paper is to investigate computational
aspects of this theorem.

\begin{prob}[Effective Deligne-Ogus-Shioda problem]\label{prob:effDOS}
  Let $g \geq 2$ be an integer.
  Given supersingular elliptic curves $E_1, \dots,  E_g$ and $E_1', \dots, E_g'$ defined over
  $\mathbb F_q$, compute an
  $\overline{\mathbb F_q}$-isomorphism $E_1\times \cdots \times E_g \to E_1'\times \cdots \times  E_g'$.
\end{prob}

This appears to be a difficult computational problem. Indeed, computing the  
endomorphism ring of a supersingular curve is a computational problem
which is considered hard, and the security of several cryptographic
constructions relies on it~\cite{2DWest, PRISM}. Solving Problem~\ref{prob:effDOS} would provide non-trivial
information about the endomorphism rings of the curves:
From an isomorphism $E_1\times E_2 \to E_1' \times E_2'$, we can compute four isogenies
$\varphi_{ij}:E_j\to E_i'$, and the composition
$\widehat \varphi_{21} \varphi_{22} \widehat \varphi_{12} \varphi_{11}:E_1 \to
E_1$ is in general a non-trivial endomorphism of $E_1$. 
In this paper, we study Problem~\ref{prob:effDOS} in the
context where we have extra information: the endomorphism rings of the elliptic curves are given (see Section~\ref{sc:endoring_encoding} for technical details on the encoding of the endomorphism rings). 
In this setting, Deuring's correspondence allows us to translate
Problem~\ref{prob:effDOS} into a problem about quaternion algebras.

\subsection{Contributions}
We focus on the case $g=2$, which is the base case which serves as a building
block for the general case $g\geq 2$.
Therefore our main problem is the computation of an isomorphism $E_1\times
E_2\to E_1'\times E_2'$ between two products of supersingular elliptic curves, assuming that their
endomorphisms rings are known. Endomorphism rings are given via an efficient
representation of a $\Z$-basis together with an explicit isomorphism with a maximal
order in the quaternion algebra $\B$.
Our main contribution is a polynomial-time algorithm that computes an isomorphism $E_1\times E_2\to E_1'\times E_2'$ between
products of maximal elliptic curves over $\F_{p^2}$, assuming that we know the endomorphism rings of the curves.
This algorithm relies on two main subroutines. The first one
describes how to build a two by two matrix of isogenies which is an isomorphism,
given its first column. The second one allows us to compute isomorphisms of
the form $E^2 \to E_1' \times E$ when we know a non-scalar low-discriminant
endomorphism in $\End(E)$.

In order to design such algorithms, we need some new computational
techniques. We provide a
quasi-linear quaternionic method to divide an endomorphism by an isogeny, see
Proposition~\ref{prop:complexityfactor}.
Our main theoretical tool is a necessary and sufficient criterion to decide
whether a pair of separable isogenies $\varphi_{11}:E_1\to E_1'$, $\varphi_{21}:E_1\to
E_2'$ of coprime degrees can appear as the first column of a matrix $(\varphi_{ij})_{i,j\in\{1, 2\}}$
describing an isomorphism $E_1\times E_2\to E_1'\times E_2'$: this happens
when the direct sum of the kernels of $\varphi_{11}$ and
$\varphi_{21}$ is the kernel of an isogeny $E_1\to E_2$. This result is formalized in
Theorem~\ref{theo:isom_iff} and is used in our algorithms for both subroutines.
In the low-discriminant case, we also use the fact that we can solve efficiently norm
equations in low-discriminant imaginary quadratic orders.
In both cases, we use Wesolowski's
variant~\cite{WesoFocs21} of the KLPT algorithm~\cite{KLPT} as an
important building block, whose complexity is proved under GRH.

% \smallskip
% The general problem of computing isomorphisms between superspecial abelian
% surfaces seems to be hard. Indeed, the techniques that we developed for
% the special cases require more degrees of freedom that what is available
% in the general case. Furthermore, randomly constructed isogenies for the
% first column of the matrix have no chance to produce a valid input for
% our criterion in Theorem~\ref{theo:isom_iff}.
% We therefore propose a cryptographic construction built on the
% difficulty of this problem. The main interest of this construction is that
% most algebraic computations can be done in the quaternion algebra $\B$ since
% the endomorphism rings and their embedding in $\B$ are public. This is a significant
% difference compared to other isogeny-based cryptographic protocols where the
% knowledge of the endomorphism ring in a quaternion algebra is usually
% sufficient to break the cryptosystem. The security of this cryptosystem relies
% on heuristic assumptions which are similar to a heuristic used foonsequentlyr the security of
% SQIsign~\cite{de2020sqisign}. In
% particular, given a product of supersingular curves $E_1\times E_2$, we can use
% the algorithm with partial control over the codomain to generate a secret
% isomorphism $E_1\times E_2\to E_1'\times E_2'$ where the pair of j-invariants
% $(j(E_1'), j(E_2'))$ is heuristically undistinguishable from the uniform
% distribution on pairs of j-invariants of supersingular elliptic curves.
% Combining this with a masking technique using automorphisms provides us with an
% authentication protocol.

To demonstrate some of the algorithms presented in this
paper, we provide a proof-of-concept implementation in the computer algebra software \texttt{Magma},
freely available at~\cite{artifact}.
%At \url{https://gitlab.inria.fr/superspecial-surfaces-isomorphisms/experiments},
%we provide a proof-of-concept implementation in the computer algebra software \texttt{Magma}, 
%which demonstrates some of the algorithms presented in this
%paper.
We only compute the quaternionic part of the isomorphisms,
meaning that we output four ideals $\{I_{ij}\}_{1\leq i,j\leq 2}$, which are
kernel ideals of four isogenies $\varphi_{I_{ij}}$, which
form a matrix that represents an isomorphism. To recover the isogenies, one can use \textsf{IdealToIsogeny} algorithms,
described for example in~\cite{2DWest,IdealToIsogeny_ON}; see also~\cite[Lem.~4]{GalPetSil17} which provides a general polynomial-time complexity bound for this step. 

\subsection{Related works} Superspecial abelian varieties are central objects
in the recent developments of \emph{isogeny-based cryptography}, as they are
the main characters of the new
high-dimensional techniques, see e.g.~\cite{robert2023breaking,HighDim_Dartois,Algo_22_isogenies}. Being able to
compute isomorphisms between such objects would be a useful computational tool.
In particular, one typical setting is to consider a special curve which has the
property that its endomorphism ring contains a low-discriminant imaginary
quadratic order. For instance, when $p\equiv 3\bmod 4$, the
endomorphism ring of the elliptic curve $E_0$ defined over $\mathbb F_{p^2}$ by
the equation $y^2=x^3+x$ contains a subring isomorphic to $\mathbb Z[i]$.
Being able to compute an isomorphism between a superspecial abelian
variety and $E_0^g$ would give access to these
low-discriminant subrings of endomorphisms. Another application of our work is
the representation of some polarized abelian varieties. 
In particular, the Ibukiyama-Katsura-Oort correspondence shows that superspecial principally polarized abelian surfaces (up to polarized isomorphisms)
can be represented via two by two matrices over the quaternions (modulo
congruence). Algorithms have been recently developed in~\cite{KLPT2} to perform efficient
computations via this representation. Our algorithms contribute to this
toolbox since isomorphisms between products of supersingular elliptic
curves can be used to compute the Ibukiyama-Katsura-Oort representation of
principally polarized abelian surfaces that are not Jacobians of genus-$2$
curves.

In the direction of the construction of isomorphisms between products of
elliptic curves, one can find in the proof of \cite[Thm~A.1]{TheoremAnnexe}
an explicit construction of an isomorphism between the $E \times E/(K_1 + K_2)$
and $E/K_1 \times E/K_2$ for $K_i$ finite étale subgroups of coprime orders.
This result has similarities with our Theorem~\ref{theo:isom_iff},
but it is not general enough for our purposes.
% The main difference is that we give a necessary and sufficient condition
% for a couple of isogenies $(\varphi_1, \varphi_2)$ to be completed into a matrix
% $(\varphi_{ij})_{1 \le i,j \le 2}$ representing an isomorphism; whereas
% in~\cite[Theorem A.1]{TheoremAnnexe} the authors only gives a sufficient
% condition, and the first raw of their matrix is not the vector $(\varphi_1, \varphi_2$)
% but a slighty twisted vector. Nevertheless their construction is more explicit.

\subsection{Organization of the paper}

Section~\ref{sec:background} recalls background material. We state our main results in Section~\ref{sc:main}. In Section~\ref{sec:tools}, we develop computational tools
that will be required in the main algorithms. Section~\ref{sec:algo_isom} is devoted to the
computation of isomorphisms between products of supersingular elliptic curves
provided that we know their endomorphism rings.

\subsection{Acknowledgements}
We thank Benjamin Wesolowski for providing the ideas that led to the
general algorithm in Section~\ref{sc:generalcase}.
We thank Jean Kieffer and Damien Robert for fruitful discussions, and the Isocrypt people for
their interest.
We thank Marc Houben for pointing us to~\cite[Thm~A.1]{TheoremAnnexe}.
This work received funding from the France 2030 program managed by
the French National Research Agency under grant agreements No. ANR-22-PETQ-0008
PQ-TLS and ANR-22-PECY-0010.
% This work received funding from the France 2030 program managed by
% the French National Research Agency under grant agreement No. ANR-22-PETQ-0008 PQ-TLS.

\section{Background}\label{sec:background}

For an elliptic curve $E$ over a field $k$, we denote
by $\End(E)$ its ring of endomorphisms defined over $\overline k$, the algebraic closure of $k$.
We assume that the characteristic $p$ of finite fields is greater than
$3$. For $p=2,3$, the algorithmic questions discussed in this paper are trivial
since all supersingular elliptic curves are isomorphic.

Throughout this paper, we use the formalism of group schemes to describe
\emph{kernels} of (non-necessarily separable) isogenies, so that any nonzero
isogeny (even if it is purely inseparable) has a non-trivial
kernel. We refer to \cite{waterhouse1969abelian} for more details on this
formalism. In particular, for an elliptic curve $E$ defined over $\overline{\F_p}$, there
are bijections between proper left-ideals
in $\End(E)$, finite group subschemes in $E$, and isogenies with domain $E$ up to
post-composition by isomorphisms. This follows from the fact that all
left-ideals in $\End(E)$ are \emph{kernel ideals}, see \cite[Thm.~3.15]{waterhouse1969abelian}
for the cases where $\End(E)$ has
rank $1$ or $4$, and \cite[Thm.~20.(a)]{kani2009products} for the CM-case.
We also use the following convenient notation: given a
finite subgroup scheme $K$ of an elliptic curve $E$, we let $E\to E/K$ denote the geometric
quotient of $E$ by $K$, where $K$ acts by translation. Therefore, an elliptic curve
$E'$ is isomorphic to $E/K$ if and only if there exists an isogeny $E\to E'$
whose kernel is $K$. We call the map $E\to E/K$ the canonical isogeny
with kernel $K$.

Elliptic curves over $\mathbb F_{p^2}$ whose number of rational points equals the Hasse-Weil
upper bound $(p+1)^2$ are called \emph{maximal}.
% \begin{defi}
%   Let $E$ be an elliptic curve over $\F_{p^2}$. We say that $E$
%   is \emph{maximal} when $\#E(\F_{p^2})=(p+1)^2$.
% \end{defi}

\begin{remq}\label{remq:Fp2model}
  Any supersingular curve $E$ over a field $k$ of characteristic $p>0$ is
  $\overline k$-isomorphic to a curve defined over $\F_{p^2}$, see~\cite[Prop.~42.1.7]{QuatAlg}.
  Therefore, for computational purposes, it is convenient to consider
  supersingular curves defined over $\F_{p^2}$. Moreover any such curve is
  $\overline{\mathbb F_p}$-isomorphic to a maximal curve $E'$
  (see~\cite[Lem.~4]{Deuring4People} and \cite[Prop.~5.1]{AbVarMax}), which
  has the convenient property that all endomorphisms and isogenies with domain
  $E$ are also defined over $\mathbb F_{p^2}$, see \cite[Lem.~5.7]{AbVarMax}.
\end{remq}

\subsection{Deuring correspondence}

We recall key concepts of the Deuring correspondence.
 For a more comprehensive study of the subject,
we refer to \cite{LerouxThesis} and \cite{QuatAlg}.

\subsubsection{Quaternion algebras.}

Let $p$ be a prime. We focus on the (unique up to isomorphism) quaternion
algebra $\B$ over $\mathbb Q$ which is ramified at $p$ and $\infty$.  The algebra
$\B$ is non-commutative, and it has dimension $4$ over $\Q$. When $p \equiv 3
\bmod 4$, a $\Q$-basis is $1,i,j,k$, where $i^2= -1$, $j^2= -p$ and $k = i\, j= - j\, i$.

Any element $\alpha \in \B$ can be encoded by coordinates $(\alpha_0, \alpha_1,
\alpha_2, \alpha_3) \in \Q^4$, such that $\alpha = \alpha_0 + \alpha_1 i  +
\alpha_2 j + \alpha_3 k$. The \emph{conjugate} of
$\alpha=\alpha_0 + \alpha_1 i+\alpha_2 j +\alpha_3 k\in \B$ is
$\overline{\alpha}=\alpha_0 - \alpha_1 i  - \alpha_2 j - \alpha_3 k$. Its
\emph{reduced trace}
is $\Trd(\alpha)=\alpha + \overline\alpha = 2\alpha_0$ and its
\textit{reduced norm} is $\Nrd(\alpha)=\alpha \cdot \overline{\alpha} =
\alpha_0^2 + \alpha_1^2 + p(\alpha_2^2+\alpha_3^2) \in \Q$.  Every nonzero
$\alpha\in\B$ is invertible, i.e., there exists a unique $\beta\in\B$ such
that $\alpha\cdot\beta = \beta\cdot\alpha = 1$. 

% We now focus on subrings
% in $\B$ involved in the Deuring correspondence:

  An \textit{order} in $\B$ is a subring which has rank $4$ as a
  $\Z$-module.
  An order is \textit{maximal} when it is not contained in a strictly larger order.
% \begin{expl}
%   Assume that $p \equiv 3 \bmod 4$.
%   A non-maximal order of $\B$ is $\Z[i,j]$. This order is contained in
%   $\Z [ i, \frac{1+k}{2}]$ \cite[Lem.~2]{KLPT}, which is maximal.
% \end{expl}
A fractional ideal $I$ is a rank-$4$ $\Z$-module in $\B$. The \emph{left and right
orders} of $I$ are defined as:
\[\mathcal{O}_{L}(I) = \{ \alpha \in \B : \alpha I \subset I \}, \quad
    \mathcal{O}_{R}(I)  = \{ \alpha \in \B : I \alpha \subset I \}.\]

    When $I \subset \mathcal{O}_{L}(I)$ (or equivalently $I \subset
    \mathcal{O}_{R}(I)$, see \cite[Lem.~16.2.8]{QuatAlg}), we say that
  $I$ is an \textit{integral ideal}.
Integral ideals in maximal orders are actually locally principal~\cite[Cor.~17.2.3]{QuatAlg}.
This implies that the completion $I \otimes \Z_{\ell}$ of an ideal $I$ in an order $\mO$ at a
prime $\ell$ unramified in $\B$ generates a principal ideal in
$\mO\otimes\Z_\ell\cong \Mat_2(\Z_\ell)$. We study localizations in more
details in Section~\ref{section:adic}. Remark that any integral ideal $I$ is a
left-$\mathcal{O}_{L}(I)$ ideal, and a right-$\mathcal{O}_{R}(I)$ ideal. In the following we will
use basic results and terminology about ideals in quaternion orders and their norms,
as exposed in~\cite[Chapter 16 and 17]{QuatAlg}.
In particular,  
the set $\Conn(\mO[1],\mO[2])$ of \emph{connecting ideals} between orders $\mO[1], \mO[2]\subset \B$ is the set of invertible lattices $I$ such that $\mathcal O_L(I) = \mO[1]$ and $\mathcal O_R(I) = \mO[2]$.

% \begin{defi}
%   Let $\mO[1]$ and $\mO[2]$ be quaternion orders in $\B$. If there exists an invertible left $\mO[1]$-ideal $I$ such that
%   $\mathcal{O}_{L}(I)=\mO[1]$ and $\mathcal{O}_{R}(I)=\mO[2]$, we say that $\mO[1]$ and $\mO[2]$ are connected.
%   In that case we call $I$ a connecting ideal between $\mO[1]$ and $\mO[2]$.
%   We denote by $\Conn(\mathcal O_1,\mathcal O_2)$ the set of all connecting ideals between $\mOi$ and $\mOd$.
% \end{defi}
% \begin{remq}
%   For maximal orders $\mO[1]$ and $\mO[2]$ in $\B$, $\Conn(\mathcal O_1,\mathcal O_2)$ is non-empty.
%   Indeed it contains $d \mO[1]\mO[2]$ where $d= [ \mO[1] : \mO[1] \cap \mO[2] ]$, which is integral.
% \end{remq}
% \begin{defi}[Ideal norm]
%  \cite[Thm.~16.1.3]{QuatAlg} 
%  Let $I\subset\B$ be an ideal. The \emph{reduced norm} of $I$ is
%  $\Nrd(I) = \gcd(\{ \Nrd(\alpha) \, : \, \alpha \in I \})$.
%  Moreover, $\Nrd(I)^2 = [ \mathcal{O}_{L}(I) \, : \, I ] 
%  = [ \mathcal{O}_{R}(I) \, : \, I ]$.
% \end{defi}
% \begin{prop} \cite[Lem.~16.3.7]{QuatAlg} 
%   Let $\mathcal O_1, \mathcal O_2, \mathcal O_3\subset\B$ be three maximal
%   orders.
%   If $I\in\Conn(\mathcal O_1, \mathcal O_2)$ and $J\in\Conn(\mathcal
%   O_2,\mathcal O_3)$, then $I\cdot J\in\Conn(\mathcal O_1, \mathcal O_3)$ and $\Nrd(I\cdot
%   J)=\Nrd(I)\cdot\Nrd(J)$.
% \end{prop}

\subsubsection{The correspondence.}
Endomorphism rings of supersingular elliptic curves are isomorphic to maximal
orders in $\B$. The purpose of Deuring correspondence is to provide a set of tools for representing
geometric objects related to supersingular elliptic curves as algebraic objects
in $\B$.

\begin{theo} \cite[Thm.~42.1.9]{QuatAlg} 
  Let $E$ be a supersingular elliptic curve defined over $\overline{\F_p}$.
  Then $\End(E)$ is isomorphic to a maximal order in $\B$.
\end{theo} 

Let $\mathcal O\subset\B$ be a maximal order isomorphic to the endomorphism
ring $\End(E)$ of a supersingular elliptic curve $E$ defined over $\overline{\F_q}$.
We will implicitly use the isomorphism $\End(E) \simeq \mO$.
There is an anti-equivalence between the category of supersingular elliptic curves over $\overline{\F_q}$
and the category of invertible left $\mathcal O$-modules. This anti-equivalence is
given explicitly via the contravariant functor $\Hom(\_, E)$, see \cite[Thm.~42.3.2]{QuatAlg}.
It establishes a dictionary
between the geometric world of supersingular elliptic curves
and the algebraic world of quaternion orders.

On the one hand, let $J$ be a left $\End(E)$-ideal. It defines a subgroup scheme
$E[J]:= \cap_{\alpha \in J}\ker \alpha$ in $E$, which is the kernel of an isogeny $\varphi_J:E \to E/E[J]$, see~\cite[42.2.1]{QuatAlg}. 
If $\varphi_J$ is separable, then
$E[J]=\{ P \in E(\overline {\F_q}) \mid \forall \alpha \in J, \, \alpha(P)=0 \}$.
On the other hand, let $\varphi: E \to E'$ be an isogeny. Then $I_\varphi:=\Hom(E',E)\varphi$ is a 
left $\End(E)$-ideal which connects the endomorphism rings of
$E$ and $E'\simeq E/\ker(\varphi)$, regarded as maximal orders in $\B$ up to
conjugation.
Moreover, for a left-ideal $J\subset\mathcal O$ and
$\psi:E \to E'$, we have that $J=I_{\varphi_{J}}$ and $\psi \cong \varphi_{I_{\psi}}$.
In particular we have a bijection between isomorphism classes (i.e., isogenies up
to post-composition by isomorphisms) of isogenies from $E$,
and left-ideals $I$ in $\mathcal O$. See \cite[Table~2.1]{LerouxThesis}
for a summary of the dictionary given by Deuring correspondence.

% \begin{table}
% \centering
%   \begin{tabular}{ | c || c |} 
%     \hline
%     Supersingular $j$-invariants over $\F_{p^2}$ & Isomorphism class of maximal order in $\B$  \\ 
    
%     $j(E)$ up to Galois conjugacy  & $\mathcal{O} \cong \End(E)$  \\ 
%     \hline
%      Isomorphism class of $\varphi:E \to E'$ & $I_{\varphi}$ integral left
%      $\mathcal O$-ideal  \\ 
%     \hline
%     $ \alpha \in \End(E)$ & principal ideal of $\mathcal O$ generated by the image of $\alpha$ \\
%     \hline

%     $\deg(\varphi)$ & $\Nrd( I_{\varphi})$ \\
%     \hline
%     $\widehat{\varphi}$ & $\overline{I_{\varphi}}$ \\
%     \hline
%     Composition $ \psi_2 \circ \psi_1: E_1 \to E_2 \to E_3$ & $I_{\psi_2 \circ \psi_1}
%                                                               =I_{\psi_1} I_{\psi_2}$ \\
%     \hline
%   \end{tabular}
%   \caption{Summary of the Deuring correspondence.\label{table:Deuring}}
% \end{table}

\subsection{Efficient representations of isogenies}

In order to implement the constructions of this work, we have to
define what is a good representation of an isogeny. 
We will use the notion of \emph{efficient representation}
designed in \cite{endringpb,EffRpz}.

  A \textit{representation} of an isogeny $\varphi: E\to E'$ between elliptic curves defined over $\F_q$, is a set of
  data that contains the domain, the codomain, the degree $\deg(\varphi)$, and
  an algorithm to evaluate $\varphi$ on any point $P\in E(\F_{q'})$ for any
  finite extension $\F_{q'}/\F_{q}$. In fact, a bound on the degree is sufficient since the degree can then be recovered using~\cite[Lem.~6.2]{EffRpz}. When $\deg(\varphi)$ is large, we may not be able to use directly accessible $N$-torsion for $N>\deg(\varphi)$; in that case, we can use a powersmooth integer for $N>\deg(\varphi)$ and a CRT representation as explained in~\cite[Appendix~A]{EffRpz}.
  We say that a representation is \textit{efficient} if this data enables us
  to compute the image of a point $P \in E(\F_{q'})$ in time polynomial in
  both $\log(\deg(\varphi))$ and $\log(q')$. We say that it is \textit{compact} if the
  space needed to store the data is polynomial in $\log(\deg(\varphi))$ and
  $\log(q')$.  
The representation we will use is the \textit{ideal representation} which relies
on the Deuring correspondence.

\subsubsection{Ideal representation.}

The core idea of the ideal representation is to represent an isogeny $\varphi:E
\to E'$ via the ideal $I_\varphi\subset\End(E)$ of all endomorphisms
whose kernel contains $\ker\varphi$, seen as an ideal in a maximal order of
$\B$ isomorphic to $\End(E)$. In order to use this representation, we first
need to fix an embedding $\End(E)\hookrightarrow \B$. Although this only encodes the isomorphism
class of $\varphi$, knowing the codomain $E'$ enables us to determine $\varphi$ up to
post-composition by automorphisms. Consequently, in order to have a full
representation of $\varphi$, we need a bit
more data to discriminate these automorphisms.
We disregard this subtlety in the present work: the order of $\Aut(E')$ is
at most $24$, so we can use exhaustive search on $\Aut(E')$ without harming the complexity. However, in a practical implementation, it might be
useful to add some information
to remove the ambiguity. Note that if $j(E') \ne 0,1728$
then $\Aut(E')=\{\pm 1\}$~\cite[Appendix A, Prop.~1.2.(c)]{Silv}.

\begin{theo} 
  Given an efficient representation of a $\Z$-basis of $\End(E)$ and its
  image via an embedding $\End(E)\hookrightarrow \B$, then a $\Z$-basis of the ideal $I_{\varphi}$ provides a compact representation of $\varphi$.
  Assuming GRH, this representation is efficient.
\end{theo}

For more details, see \cite[Sec.~4.2 and C.1]{EffRpz}.
We work with $2\times 2$ matrices whose entries are
isogenies, which can be encoded via efficient representations.

\begin{prop}
  Let $E_1, E_2, E_1', E_2', E_1'', E_2''$ be elliptic curves defined over
  $\F_q$. Let $M = (\varphi_{ij})_{i,j\in\{1, 2\}}$ (resp. $N =
  (\psi_{ij})_{i,j\in\{1, 2\}}$) be a $2\times 2$ matrix of isogenies, where
  $\varphi_{ij}:E_j\to E_i'$ (resp. $\psi_{ij}:E_j'\to E_i''$). Then $M$ (resp.
  $N$) represents the isogeny $E_1\times E_2\to E_1'\times E_2'$ (resp.
  $E_1'\times E_2'\to E_1''\times E_2''$) defined as
  $\phi_M(P,Q)=(\varphi_{11}(P) + \varphi_{12}(Q), \varphi_{21}(P) +
  \varphi_{22}(Q))$ (resp. $\phi_N(P,Q)=(\psi_{11}(P) + \psi_{12}(Q), \psi_{21}(P) +
  \psi_{22}(Q))$). Moreover, the matrix product $N\cdot M =
  (\sum_{k\in\{1, 2\}} N_{ik}\circ M_{kj})_{i,j\in\{1, 2\}}$ represents an
  isogeny $E_1\times E_2\to E_1''\times E_2''$ and efficient representations of
  the entries of $N\cdot M$ can be computed in polynomial-time from efficient
  representations of the entries of $M$ and $N$.
\end{prop}

\begin{proof}
  The only thing that we need to prove is that we can compute efficient
  representations of compositions and sums of isogenies encoded with efficient
  representations. Algorithms for doing so are described in \cite[Sec.~6.1]{EffRpz}.
\end{proof}

% \subsubsection{HD representation.}
% The successive attacks on SIDH have led to new constructive applications. One of the most
% significant for us is that we can now evaluate an isogeny using calculations performed in
% higher dimensions, given only the image of certain torsion points. And to work with such points,
% we may need to deal with high degree extensions of $\F_q$. But to keep efficient computation we
% can not work in too big finite fields.{}
% \begin{theo} \cite[Thm.~5.19]{EffRpz}
%   Let $\varphi:E \to E'$ be an isogeny of degree $n$. Let $N=\prod \ell_i  > n$ be a smooth
%   integer coprime to $n$, with $\max(\ell_i)=\log^{O(1)}N$. For each $i$, let $(P_i,Q_i)$ be
%   a $\Z$-basis of $E[\ell_i]$, such that $\langle \bigoplus\limits_{i} P_i ,
%   \bigoplus\limits_{i} Q_i \rangle \, = \, E[N]$.
%   The data of $n, E, E'$ and the interpolation data $(P_i,\varphi(P_i),Q_i,\varphi(Q_i))_i$ gives a compact and
%   efficient representation of $\varphi$, called an HD representation.
% \end{theo}

% Note that this representation is \textit{universal}, meaning that any efficient
% representation can be efficiently converted into an HD representation. An
% interesting feature for cryptographic applications is that the interpolation
% data does not reveal any information about the way the isogeny was
% constructed.

\subsubsection{Knowing the endomorphism ring of a curve.}\label{sc:endoring_encoding} Throughout this
paper, we often say that the endomorphism ring of a supersingular elliptic
curve $E$ is ``known'' or ``given''. By this, we mean that efficient representations of a
$\Z$-basis $b_1,\ldots, b_4$ of $\End(E)$ is given, and that we also have
access to elements $\beta_1,\ldots, \beta_4\in\B$ such that the
$\Z$-module $\mathcal O$ generated by $\beta_1,\ldots, \beta_4$ in $\B$ is a maximal order
and the map $\End(E)\to\mathcal O$ sending $b_i$ to $\beta_i$ is a ring
isomorphism.

\subsection{Superspecial Abelian varieties}

The key theoretical result we rely on is the following statement of existence.

\begin{theo}\label{theo:Superspecial}(Deligne-Ogus-Shioda theorem) \cite[Thm.~3.5]{DelignTheo}
  Let $k$ be an algebraically closed field of characteristic $p > 0$. Let $E_1, \dots, E_g$ and 
  $E_1', \dots, E_g'$ be supersingular elliptic curves, where $g \ge 2$.
  Then there exists an isomorphism:
  $E_1 \times \dots \times E_g \cong    E_1' \times \dots \times E_g'.$
\end{theo}

\begin{defi}
  An abelian variety $\mathcal{A}$ is called \textit{superspecial} when it is isomorphic to a product
  of supersingular elliptic curve.
\end{defi}

In other words, Theorem \ref{theo:Superspecial} states that there is only one superspecial
abelian variety of each dimension $g \ge 2$, up to isomorphism. We emphasize that we do not take
into account the \textit{polarizations} of the abelian varieties.

When $p\equiv 3\bmod 4$, there is a supersingular elliptic curve
defined over $\mathbb F_p$ by the equation $y^2 = x^3 + x$. We denote this
special curve by $E_0$ throughout this paper. A useful feature of
$E_0$ is that $\End(E_0)$ contains a subring isomorphic
to $\Z[i]$. In the case $p\equiv 1 \bmod 4$, a similar curve whose endomorphism ring contains a low-discriminant imaginary quadratic subring can be computed efficiently~\cite[Sec.~3.1]{Deuring4People}. A direct consequence of
Deligne-Ogus-Shioda theorem is that any superspecial variety of dimension $g$
defined over $\overline{\F_p}$ is $\overline{\F_p}$-isomorphic to $E_0^g$.

% \begin{remq}
%   Theorem~\ref{theo:Superspecial} does not hold true for $g=1$. For instance over the field $\F_{419^2}$, the supersingular curve defined by
%   $E : y^2 + xy = x^3 + 164x +377$ is not isomorphic to $E_0$.
%   However, $E_0^2 \cong E^2$.
% \end{remq}

If $E_1$, $E_2$, $E_1'$, $E_2'$ are supersingular elliptic curves defined over
$\F_{p^2}$, Theorem~\ref{theo:Superspecial} implies that $E_1\times E_2$ and
$E_1'\times E_2'$ are $\overline{\F_p}$-isomorphic.
In fact, when $E_1, E_2, E_1', E_2'$ are maximal, this isomorphism
is defined over $\F_{p^2}$, see~\cite[Lem.~5.2]{AbVarMax}. 

In this work, we explore the problem of computing explicit isomorphisms between
products of supersingular elliptic
curves. Moreover, as explained in Remark~\ref{remq:Fp2model}, we can
choose maximal models for the supersingular elliptic curves we work with.
The goal of this article is thus to find an
$\F_{p^2}$-isomorphism between the products $E_1 \times \cdots \times E_g$ and $E_1' \times \cdots \times E_g'$,
where the curves are maximal over $\F_{p^2}$.

\begin{remq}\label{remq:NonMaxCase}
  If the supersingular input curves are not maximal, we can start
  by computing $\overline{\F_p}$-isomorphic maximal curves. This can be done
  efficiently and it does not have any impact on our complexity bounds,
  see~\cite[Lem.~4]{Deuring4People}.
%  Let $E_1,E_2,E_1'$ and $E_2'$ be supersingular elliptic curves over $\F_{p^2}$.
%  As explained in Remark~\ref{remq:Fp2model}, there exists maximal  ${E_i}^{max},{E_i'}^{max}$ of these curves
%  along with $\F_q$-isomorphisms $\psi_i:E_i \to {E_i}^{max}$ and $\psi_i':{E_i'}^{max} \to {E_i'}$, where $\F_q$
%  is a finite extension of $\F_{p^2}$. Indeed there is only four isomorphisms to define, so we can choose
%  $\F_q$ to be finite in such a way that the maximal models are also defined over $\F_q$.
%  Now if $\Phi: {E_1}^{max} \times {E_2}^{max} \to {E_1'}^{max} \times {E_2'}^{max}$ is an $\F_{p^2}$ isomorphism,
%  then the composition $\diag(\psi_1',\psi_2') \circ \Phi \circ \diag(\psi_1,\psi_2): E_1 \times E_2 \to E_1' \times E_2'$
%  is an $\F_q$-isomorphism.
%  This generalization comes at the cost of slower arithmetic (working over $\F_q$ instead of $\F_{p^2}$),
%  and the computation of the $\psi,\psi'$.
\end{remq}

\section{Main Results}\label{sc:main}

\begin{algorithm}
  \caption{\ProductIsomorphism{}}\label{algo:ProductIsomorphism}
    \INPUT{ Maximal elliptic curves $E_1,\dots,E_g,E_1',\dots,E_g'$ over $\F_{p^2}$ for $g \geq 2$; $\Z$-bases of maximal orders $\mO[i],
     \mO[i]' \subset \B$ such that $\mO[i] \simeq \End(E_i)$ and $\mO[i]' \simeq \End(E_i')$ for $i \in \{1,\dots, g\}$.
  }
    \OUTPUT{An efficient representation of a $g\times
    g$ matrix of isogenies ($\varphi_{ij}$) representing an isomorphism
    $E_1\times \cdots \times  E_g \to
    E_1'\times \dots \times E_g'$}

  \eIf{$g=2$}{
    Return $\ProductSurfacesIsomorphism(E_1,E_1',E_2,E_2',\mO[1],\mO[1]',\mO[2],\mO[2]')$
    }{
   Compute an isomorphism $\Phi_{g-1}:E_1\times\dots\times E_{g-1}\to E_1'\times \dots \times E_{g-1}'$ by 
   a recursive call to $\ProductIsomorphism{}$ \;
   % Deduce an isomorphism $\Phi_g:E_1\times\dots\times E_g\to E_1'\times \dots \times E_{g-1}' \times E_g$ 
   % given by $(\Phi_{g-1}, Id_{E_g})$ \;
   Compute an isomorphism of surfaces $\varphi: E_{g-1}' \times E_g \to E_{g-1}' \times E_g'$ 
   by $\ProductSurfacesIsomorphism$ \;
   % Deduce an isomorphism $\Psi_g:E_1' \times \cdots \times  E_{g-1}' \times E_g \to E_1'\times \dots \times E_g' $
   % given by $( Id_{ E_1' \times \cdots \times E_{g-2}'}, \varphi)$ \;
   Return $(Id_{ E_1' \times \cdots \times E_{g-2}'}, \varphi ) \circ (\Phi_{g-1}, Id_{E_g})$.
  }
\end{algorithm}

Our main result is Algorithm~\ref{algo:ProductIsomorphism} whose
properties are as follows.

\begin{theo}\label{theo:main}
  Let $E_1, \dots, E_g$ and $E_1', \dots, E_g'$ be maximal elliptic curves over $\F_{p^2}$
  with known endomorphism rings, where $g \ge 2$. Algorithm~\ref{algo:ProductIsomorphism} is a probabilistic Las Vegas
  algorithm which computes  an $\F_{p^2}$-isomorphism of unpolarized abelian varieties:
  $E_1 \times \dots \times E_g \cong    E_1' \times \dots \times E_g'$.
   It runs in time linear in $g$, and polynomial in $\log(p)$ under GRH.
\end{theo}

In fact, the key point is Algorithm~\ref{algo:ProductSurfacesIsomorphism}
which deals with the dimension-2 case, for which we have the following
result.

\begin{theo}\label{theo:main2}
  Let $E_1,E_2$ and $E_1', E_2'$ be maximal elliptic curves over $\F_{p^2}$ with
  known endomorphism rings. Algorithm~\ref{algo:ProductSurfacesIsomorphism} is a probabilistic Las Vegas
  algorithm which computes  an $\F_{p^2}$-isomorphism of unpolarized abelian surfaces:
 $E_1 \times E_2 \cong    E_1' \times E_2'$.
  Assuming GRH, it runs in time polynomial in $\log(p)$.
\end{theo}

The rest of the paper will be dedicated to the proof of Theorem~\ref{theo:main2}.
Let us explain how to deduce Theorem~\ref{theo:main} from Theorem~\ref{theo:main2}.

\begin{proof}[Proof of Thm.~\ref{theo:main} from Thm.~\ref{theo:main2}]
  Here we assume that Algorithm~\ref{algo:ProductSurfacesIsomorphism} is correct.
  Under this assumption we provide a proof by induction on $g$ that Algorithm~\ref{algo:ProductIsomorphism}
  returns the desired isomorphism.
  By Deligne-Ogus-Shioda theorem, an isomorphism $E_1\times\dots\times E_g\to
  E_1'\times\dots\times E_g'$ can be factored as a composition
  $E_1\times \dots \times E_g\to E_1'\times \dots\times E_{g-1}'\times E_g \to
  E_1'\times \dots\times E_g'.$

  This boils down to computing one $(g-1)$-dimensional isomorphism $E_1\times\dots\times E_{g-1}\to
  E_1'\times \dots \times E_{g-1}'$ and one $2$-dimensional isomorphism
  $E_{g-1}'\times E_g\to E_{g-1}'\times E_g'$. By induction, this proves that
  computing a $g$-dimensional isomorphism can be reduced to computing $g-1$
  isomorphisms in dimension $2$. Finally, since Algorithm~\ref{algo:ProductSurfacesIsomorphism} runs
  in time polynomial in $\log(p)$ under GRH, the complexity statement follows.

\end{proof}

\section{Tools}\label{sec:tools}

In this section, we develop tools which will be useful for computing
isomorphisms in Section~\ref{sec:algo_isom}. In Section~\ref{sc:div_ideal}, we study algorithms
for dividing endomorphisms by isogenies. Section~\ref{sec:improv_Kani} proves a
slight improvement of Kani's formula for the degree of an isogeny between
products of elliptic curves; this is useful for proving that an isogeny is an
isomorphism. In Section~\ref{sec:transposing_isog}, we show how to
``transpose''
isogenies between products of elliptic curves: we provide an easy way to
construct an isogeny $E_1'\times E_2'\to E_1\times E_2$ from an isogeny
$E_1\times E_2\to E_1'\times E_2'$, while preserving the degree.
% In Section~\ref{sec:automorphisms}, we describe families of easily constructible
% automorphisms of products of elliptic curves; these automorphisms are useful
% for hiding secret data in cryptographic constructions, see
% Section~\ref{sc:crypto}.
Finally, in Section~\ref{section:adic}, we design
algorithms for finding the generator of the localization of a left-ideal in a
maximal order of $\B$.

\subsection{Division of principal ideals in quaternion orders}\label{sc:div_ideal}

The first tool that we need is a method for dividing efficiently an endomorphism by an
isogeny. More precisely, given an endomorphism $\phi\in\End(E_1)$, which
factors by an
isogeny $f:E_1\to E_2$, we wish to compute an isogeny $g:E_2\to E_1$ (which is
uniquely defined up to composition by automorphisms) such that
$\phi = g\circ f$. 
A general method when isogenies are given via efficient representations is
described in \cite[Cor.~6.8]{EffRpz}. A detailed complexity analysis is
provided in \cite[Sec.~4]{merdy2023supersingular} when $g$ is a scalar multiplication. We
propose here an explicit complete algebraic solution to the quaternionic
version of the problem, i.e., when all isogenies are represented as ideals in
$\B$.
This factorization problem is formalized in quaternion algebras as follows:

\begin{prob}[Principal ideal division]\label{pb:div_prin}
  Let $\mathcal O_1, \mathcal O_2, \mathcal \subset \B$ be two  maximal orders.
  Let
  $\mu\in\mathcal O_1$, $I\in \Conn(\mathcal O_1, \mathcal O_2)$, and $J$ be a
  left $\mathcal O_2$-ideal
  such that $\mathcal O_1\mu = I\cdot J$. Given $\mu$ and 
  $\Z$-bases of $\mathcal O_1, \mathcal O_2, I$, find a $\Z$-basis of $J$.
\end{prob}

% i dont find the lemma i quoted...
% \begin{remq}
%   If $\Nrd(I)$ and $\Nrd(J)$ are coprime, then \cite[Lem.~6]{2DWest} allows
%   us to recover $I$ more easily. However here we need to compute $J$,
%   and the assumption that $\Nrd(I)$ and
%   $\Nrd(J)$ are coprime is too strong for our setting: in theory (and in
%   experiments), this hypothesis is not always satisfied. Therefore,
%   we design a general algorithm which does not require any such
%   assumption on the input.
% \end{remq}

First we remark that Problem \ref{pb:div_prin} is unambiguous.
\begin{lemm}\label{lemm:unique_ideal}
  The left $\mathcal O_2$-ideal $J$ in Problem~\ref{pb:div_prin} is uniquely determined.
\end{lemm}
\begin{proof}
  Let $B_1$ and $B_2$ be two solutions of Problem~\ref{pb:div_prin}, and respectively
  denote by $J_1$ and $J_2$ the ideals they generate.
  Then we have $I\cdot J_1 = I\cdot J_2.$ By multiplying on the left by $\overline I$,
  we obtain that $\Nrd(I)\cdot\mathcal O_R(I)\cdot J_1 = \Nrd(I)\cdot\mathcal
  O_R(I)\cdot J_2$, see~\cite[Sec.~16.6]{QuatAlg}.
  Moreover $\mathcal{O}_R(I)=\mathcal{O}_L(J_1)= \mathcal{O}_L(J_2)=\mathcal{O}_2$,
  and $J_1$, $J_2$ are left-ideals in $\mathcal{O}_2$. Therefore, $\Nrd(I)\cdot
  J_1 = \Nrd(I)\cdot J_2$, which implies $J_1=J_2$.
\end{proof}

We first address the case where $\mu \in \Z$. In fact, we will show in Remark~\ref{remq:unicity_int_ideal_div} that the general case $\mu\in\mathcal O_1$ can be reduced to the case where $\mu\in\Z$.

\begin{prob}[Integer ideal division]\label{pb:int_ideal_div}
  Let $\mathcal O_1, \mathcal O_2, \mathcal \subset \B$ be two  maximal orders. Let
  $d \in \Z$, $I\in \Conn(\mathcal O_1, \mathcal O_2)$, and $J$ a left
  $\mOd$-ideal be
  such that $\mathcal O_1 d = I\cdot J$. Given 
  $d$ and $\Z$-bases of $\mathcal O_1, \mathcal O_2, I$, find a $\Z$-basis of $J$.
\end{prob}

\begin{remq}\label{remq:unicity_int_ideal_div}
  The same argument as in the proof of Lemma~\ref{lemm:unique_ideal} shows that
  Problem~\ref{pb:int_ideal_div} is well defined. Moreover we can swap the roles of
  $I$ and $J$ via conjugation since $I \cdot J= d \mOi = \overline J  \cdot
  \overline I = d \mOi$. It is easy to check that
  Problems~\ref{pb:int_ideal_div} and~\ref{pb:div_prin} are equivalent: the
  solution $J$ of Problem~\ref{pb:div_prin} with input
  $\mO[1], \mO[2], \mu, I$ equals the solution of Problem~\ref{pb:div_prin}
  with input $\mu^{-1}\mO[1]\mu$, $\mO[2]$, $\Nrd(\mu)$, $\overline\mu I$.
\end{remq}

Now we propose an efficient method to solve Problem \ref{pb:int_ideal_div}.

\begin{prop}\label{prop:J=S}
  With the same notation as in Problem~\ref{pb:int_ideal_div}, $J = \{
    s\in\mathcal O_1 \cap \mathcal O_2 :  I s  \subset \mathcal O_1d \}$.
\end{prop}

\begin{proof}
  Set $S := \{
    s\in\mathcal O_1 \cap \mathcal O_2 :  I s  \subset \mathcal O_1d \}$.
  We must show that $S$ is a left-ideal in $\mathcal O_2$
  and that $IS=\mOi d$.
  First we show that $S$ is a left-ideal of $\mathcal{O}_2$. Let $s \in S, x
  \in \mathcal{O}_2$. First, we notice that
  $I$ is a right-ideal in $\mathcal O_2$,
  thus $Ix\subset I$. Since $s \in S$, we get $Ixs\subset Is\subset \mOi d$,
  hence $xs\in S$. 
  Finally, we prove that $IS=\mOi d$. Notice that $I  S\subset \mathcal O_1 d$ by construction.
  In order to prove the other inclusion, we notice that $J$ is included in
  $S$; hence, $\mathcal O_1 d = I  J \subset I S$.
\end{proof}

Proposition~\ref{prop:J=S} reduces Problem~\ref{pb:int_ideal_div} to
$\Z$-linear algebra. Let $(e_0,\dots,e_3)$, $(u_0,\ldots, u_3)$, $(v_0,\ldots,
v_3)$ be $\Z$-bases of $\mO[1], I, \mO[1]\cap \mO[2]$ respectively.
We need to solve the following system over the integers of $4$
equations in $20$ unknowns
$\{x_i\}_{0\leq i\leq 3}$, $\{y_{ij}\}_{0\leq i,j\leq 3}$:
$$(E_j):u_j\sum_{0\leq i\leq 3} x_i v_i = d\sum_{0\leq i\leq 3} y_{ij} e_i.$$

Let $b^{(1)}, \ldots, b^{(16)}\in\Z^{20}$ be a $\Z$-basis of the solutions of
this system. Then, writing $b^{(i)} = (x^{(i)}_0,\ldots,
x^{(i)}_3,y^{(i)}_{00},\ldots, y^{(i)}_{33})$, we compute a basis $a^{(1)},
\ldots, a^{(3)}$ of the lattice generated by $\{(x^{(i)}_0, \ldots,
x^{(i)}_3)\}_{1\leq i\leq 16}$. Finally, writing $a^{(j)}=(a^{(j)}_0,\ldots,
a^{(j)}_3)$, the set $\{\sum_{0\leq i\leq 3} a^{(j)}_i v_i\}_{0\leq j\leq 3}$ is
a $\Z$-basis for $J$.

\begin{prop}\label{prop:complexityfactor}
  With the same notation as in Problem~\ref{pb:int_ideal_div}, let $\gamma$ be
  the maximum absolute value of the numerators and denominators of the coefficients of the
  elements in the bases of $\mO[1], \mO[2], I$, when written in the canonical
  basis $1, i, j, ij$ of $\B$. Then a $\Z$-basis of $J$ (written in the basis
  $1,i,j,ij$) can be computed in quasi-linear complexity $\widetilde
  O(\log\gamma)$.
\end{prop}

\begin{proof}
  A $\Z$-basis for $J$ is obtained via linear algebra over $\Z$ from
  the input $\Z$-bases. It can be computed via a sequence of Hermite Normal
  Forms
  of matrices with dimensions bounded above by a constant. Our proposition follows from the fact
  that the Hermite Normal Form of a nonzero matrix
  $(A_{ij})$ can
  be computed with complexity quasi-linear in $\max_{ij}(\log\lvert
  A_{ij}\rvert)$, see~\cite[Chap.~6]{storjohann2000algorithms}.
\end{proof}

\begin{remq}\label{remq:pb_red_facto}
  The reduction in Remark~\ref{remq:unicity_int_ideal_div} shows that Problem~\ref{pb:div_prin} can
  also be solved in quasi-linear complexity.
\end{remq}

\subsection{A corollary of Kani's formula for the degree of isogenies
  between products of elliptic curves}\label{sec:improv_Kani}

The following statement is a slight improvement of a formula due to Kani,
for the degree of an isogeny between products of elliptic curves.
Our formula from Proposition~\ref{prop:KaniDeg} reveals the sign of a quantity
appearing in Kani's work ~\cite[Cor.~64]{KaniJacobian}.
In the following statement, we use the convention that the zero morphism,
that is not an isogeny, has degree $0$.

\begin{prop}\label{prop:KaniDeg}
  Let $E_1,E_2,E_1',E_2'$ be elliptic curves defined over $\overline{\F_p}$.
 For $i,j\in\{1,2\}$, let $\varphi_{ij}\in\Hom(E_j, E_i')$ be a morphism of
  degree $d_{ij}\in\N$. Let $\phi\in\Hom(E_1\times E_2, E_1'\times E_2')$ be
  the morphism defined as $\phi(x_1, x_2) =
  (\varphi_{11}(x_1)+\varphi_{12}(x_2), \varphi_{21}(x_1)+\varphi_{22}(x_2))$.
  Then
  $\deg(\phi)
    =
    (d_{11}+d_{21})(d_{12}+d_{22}) -
    \deg(\widehat\varphi_{12}\varphi_{11} +
    \widehat\varphi_{22}\varphi_{21}).$
\end{prop}

\begin{proof}
  Set $\mu:= \widehat\varphi_{12}\varphi_{11}$ and $\nu :=
\widehat\varphi_{22}\varphi_{21}$. \cite[Cor.~64]{KaniJacobian} states that
$\deg(\phi) = \lvert(d_{11}+d_{21})(d_{12}+d_{22}) -
    \deg(\mu + \nu)\rvert$.
    Therefore, the only thing that we need to prove is that  $\deg(\mu + \nu)\leq
    (d_{11}+d_{21})(d_{12}+d_{22})$, so that the absolute value is not
    required.
We start with the following computation:
\[\begin{array}{rcl}
    0&\leq& \deg(d_{21}\mu-d_{11}\nu)\\
  &=&(d_{21}\mu - d_{11}\nu)(d_{21}\widehat\mu - d_{11}\widehat\nu)\\
  &=&d_{21}^2 \deg(\mu) + d_{11}^2 \deg(\nu)-d_{11}d_{21}(\nu\widehat\mu +
  \mu\widehat\nu).
\end{array}\]

Next, we notice that $\deg(\mu + \nu) -\deg(\mu)-\deg(\nu) = (\mu+\nu)(\widehat\mu+\widehat\nu)-\deg(\mu)-\deg(\nu)  =
\nu\widehat\mu +
\mu\widehat\nu$. Replacing $\nu\widehat\mu +
\mu\widehat\nu$ in the previous inequality, we obtain 
\[d_{21}^2 \deg(\mu) + d_{11}^2 \deg(\nu)\geq d_{11}d_{21}(\deg(\mu +
\nu)-\deg(\mu)-\deg(\nu)).\]
We replace $\deg(\mu)$ and $\deg(\nu)$ by their respective values
$d_{12}d_{11}$ and $d_{22}d_{21}$ to obtain 
\[d_{21}^2 d_{12}d_{11} + d_{11}^2 d_{22}d_{21}\geq d_{11}d_{21}(\deg(\mu +
\nu)-d_{12}d_{11}-d_{22}d_{21}).\]
Dividing this inequality by $d_{11}d_{21}$ gives $\deg(\mu + \nu)\leq
(d_{11}+d_{21})(d_{12}+d_{22})$.
\end{proof}

We shall use Proposition~\ref{prop:KaniDeg} in order to compute degrees of
isogenies between products of elliptic curves. In this usecase, we know efficient representations for $\varphi_{ij}$, and we want to compute the degree of $\phi$.
This requires computing $\deg(\widehat\varphi_{12}\varphi_{11} +
    \widehat\varphi_{22}\varphi_{21})$, which can be obtained from the techniques described in \cite[Lem.~6.2 and Appendix~A]{EffRpz} and from the bound  $\deg(\widehat\varphi_{12}\varphi_{11} +
    \widehat\varphi_{22}\varphi_{21})\leq (d_{11}+d_{21})(d_{12}+d_{22})$.

 An important case is
that it can be used to check if such an isogeny is
an isomorphism.
More precisely, a $2$-dimensional isogeny 
can be given as a matrix of isogenies $(\varphi_{ij})=\begin{pmatrix}
   \varphi_{11} & \varphi_{12} \\
   \varphi_{21} & \varphi_{22} \\           
 \end{pmatrix}$. It is an isomorphism if and only if
$(d_{11}+d_{21})(d_{12}+d_{22}) -
    \deg(\widehat\varphi_{12}\varphi_{11} +
    \widehat\varphi_{22}\varphi_{21}) =1.$
By using this property, we obtain:
 \begin{prop} \label{prop:isomorphism_suff_cond}
   Let $E_1, E_2, E_1', E_2'$ be four elliptic curves defined over
   $\overline{\F_p}$, and
   $\varphi_{ij}:E_j\to E_i'$, $i,j\in\{1, 2\}$ be four isogenies. 
   Set $\mu=\widehat\varphi_{12}\varphi_{11}, \nu =
   \widehat\varphi_{22}\varphi_{21}$, and write $d_{ij} = \deg(\varphi_{ij})$.
   Then $\deg( d_{21} \mu - d_{11} \nu) =  d_{11} d_{21}$ if and only if $\phi=(\varphi_{ij})_{i,j\in\{1,2\}}\in\Hom(E_1 \times E_2, E_1' \times E_2')$
  is an isomorphism.
\end{prop}

\begin{proof}
By Proposition~\ref{prop:KaniDeg}, we have
  $\deg(\mu+\nu) = (d_{11}+d_{21})(d_{12}+d_{22}) -
  \deg(\phi)$.
  Therefore, we obtain the equality
  \begin{equation}\deg(\mu+\nu) - \deg(\mu) - \deg(\nu) = \frac{d_{11}}{d_{21}} \deg(\nu) +
    \frac{d_{21}}{d_{11}} \deg(\mu) - \deg(\phi).\label{eq:normeq2endo}\end{equation}
  By multiplying~\eqref{eq:normeq2endo} by $d_{11}d_{21}$, we obtain
  \begin{align*}
    d_{11} d_{21} \deg(\phi) &= d_{11}(d_{11}+d_{21})\deg(\nu)+
    d_{21}(d_{11}+d_{21})\deg(\mu)-d_{11}d_{21}\deg(\mu+\nu)\\
    &= d_{11}^2\deg(\nu)+d_{21}^2\deg(\mu)-d_{11}d_{21}\Trd(\mu\widehat\nu)\\
    &= \deg( d_{21} \mu - d_{11} \nu),
  \end{align*}
  hence $\deg(\phi)=1$ if and only if $\deg( d_{21} \mu - d_{11} \nu) =
  d_{11}d_{21}$.
\end{proof}

\subsection{Transposing isogenies}\label{sec:transposing_isog}

In this section, we show how an isogeny $\phi: E_1\times E_2\to E_1'\times E_2'$ can
be transformed into a transposed isogeny $\widetilde\phi: E_1'\times E_2'\to
E_1\times E_2$ of the same degree.
Since we have not fixed
any polarization on the product surface, this transposed isogeny is
not a dual of $\phi$ in the usual sense. In particular, the composed
endomorphism
$\widetilde\phi\cdot\phi$ need not be the multiplication by an integer. Still,
the degree is preserved, i.e., $\deg(\phi)=\deg(\widetilde\phi)$.

\begin{coro}\label{coro:pseudodual}
  With the same notation as in Proposition~\ref{prop:KaniDeg}, let $\widetilde
  \phi\in\Hom(E_1'\times E_2', E_1\times E_2)$ denote the morphism defined as
  $\widetilde\phi(x_1', x_2') = (\widehat\varphi_{11}(x_1') +
    \widehat\varphi_{21}(x_2'),
  \widehat\varphi_{12}(x_1')+\widehat\varphi_{22}(x_2'))$, i.e., in
  matrix notation
  $\widetilde\phi = \begin{bmatrix}
      \widehat\varphi_{11}& \widehat\varphi_{21}\\ 
      \widehat\varphi_{12}& \widehat\varphi_{22}
  \end{bmatrix}$.
    Then $\deg(\phi) = \deg(\widetilde\phi)$.
\end{coro}

\begin{proof}
  Set $d_{ij}=\deg(\varphi_{ij})$ and $\psi:=\varphi_{21}\widehat\varphi_{11}  + \varphi_{22}\widehat\varphi_{12}$, then
$$\phi \widetilde\phi
      = \begin{pmatrix}
        (d_{11} + d_{12}) & \widehat{\psi} \\
        \psi           & (d_{21} + d_{22})\\
      \end{pmatrix}.
$$

    Applying Proposition~\ref{prop:KaniDeg} to the composed endomorphism $\phi
    \widetilde\phi$, we get 
    $$
    \deg(\phi)\deg(\widetilde\phi) = \deg ( \phi \widetilde\phi ) %
      = \left(  (d_{11} + d_{12})(d_{21} + d_{22}) - \deg( \psi) \right)^2
     = \deg( \widetilde\phi ) ^2 .
    $$
Therefore, $\deg( \phi ) = \deg( \widetilde\phi ).$
\end{proof}

\subsection{Localization}\label{section:adic}

In this section, we investigate algorithmic aspects of the ring $\Mat_2(\Z_\ell)$
and of its left-ideals. This is useful for studying localizations
of quaternion algebras: when $\mathcal O$ is a maximal order in a quaternion
algebra over $\Q$ not ramified at $\ell$, then $\mathcal O\otimes \Z_\ell$ is
isomorphic to $\Mat_2(\Z_\ell)$. The first thing to notice is that $\Mat_2(\Z_\ell)$ is
left-principal, and its left-ideals correspond to matrices in Hermite Normal
Form.

\begin{prop}\cite[Chap.~II, Thm.~2.3]{QuatVigneras}
  The left-ideals in $\Mat_2(\Z_\ell)$ are the (all distinct) ideals of the form 
  $$\Mat_2(\Z_\ell)\cdot \begin{pmatrix}\ell^n & r\\0&\ell^m\end{pmatrix},$$
    where $n,m\in\N$ are nonnegative integers, and $r\in\{0,\ldots, \ell^{m-1}\}$.
\end{prop}

As $\Mat_2(\Z_\ell)$ is left-principal, we can define the \emph{right-gcd} of
matrices $A_1, A_2\in \Mat_2(\Z_\ell)$ as the Hermite Normal Form of a
generator of the ideal $\Mat_2(\Z_\ell)\cdot A_1+\Mat_2(\Z_\ell)\cdot A_2$. We
now consider the problem of computing this right-gcd.

\begin{prop}\label{prop:rgcd_mat_zl}
  Let $A_1, A_2\in\Mat_2(\Z_\ell)$ be two matrices in Hermite Normal Form:
  $$A_i = \begin{pmatrix}\ell^{n_i} & r_i\\
  0&\ell^{m_i}\end{pmatrix}, \quad i\in\{1, 2\}.$$
  We assume without loss of generality that $n_2\geq n_1$. Set $m=\min(m_1,
  m_2, \val_\ell(r_2-\ell^{n_2-n_1}r_1))$ (with the convention that
  $\val_\ell(0) = \infty$). Then the
  \emph{right-gcd} of $A_1$ and $A_2$ is
  $$\rgcd(A_1, A_2) = \begin{pmatrix}\ell^{n_1}&(r_1\bmod
  \ell^m)\\0&\ell^m\end{pmatrix}.$$
\end{prop}

\begin{proof}
  We have to prove that 
  $$\Mat_2(\Z_\ell)\cdot A_1 + \Mat_2(\Z_\ell)\cdot A_2 = \Mat_2(\Z_\ell)\cdot
  \begin{pmatrix}\ell^{n_1}& (r_1\bmod \ell^m)\\0&\ell^m\end{pmatrix}.$$
  We notice that the left-ideal generated by a matrix corresponds to the
  $\Z_\ell$-module generated by its rows.

  First we prove the inclusion
$$\Mat_2(\Z_\ell)\cdot A_1 + \Mat_2(\Z_\ell)\cdot A_2 \supset \Mat_2(\Z_\ell)\cdot
  \begin{pmatrix}\ell^{n_1}& (r_1\bmod \ell^m)\\0&\ell^m\end{pmatrix}.$$
    The vector $(0, \ell^m)$ clearly belongs to the $\Z_\ell$-module generated
    by the rows of $A_1$ and $A_2$ since $(0, \ell^{m_1}), (0, \ell^{m_2})$ and
    $(0, r_2-\ell^{n_2-n_1} r_1)$ belong to it. Hence,
    $(\ell^{n_1}, r_1\bmod \ell^m)$ also lies in this $\Z_\ell$-module.

    Let us now prove the other inclusion:
$$\Mat_2(\Z_\ell)\cdot A_1 + \Mat_2(\Z_\ell)\cdot A_2 \subset \Mat_2(\Z_\ell)\cdot
  \begin{pmatrix}\ell^{n_1}& (r_1\bmod \ell^m)\\0&\ell^m\end{pmatrix}.$$
    The only non-trivial thing that we need to prove is that $(\ell^{n_2}, r_2)$ belongs to
    the $\Z_\ell$-module generated by $(\ell^{n_1}, r_1)$ and $(0,
    \ell^m)$. We notice that $\val_\ell(r_2-\ell^{n_2-n_1}r_1)\geq m$, hence
    there exists $x\in\Z_\ell$ such that $r_2-\ell^{n_2-n_1}r_1=x\, \ell^m$.
    Therefore $(\ell^{n_2}, r_2) = \ell^{n_2-n_1}\cdot (\ell^{n_1}, r_1) +
    x\cdot (0, \ell^m)$, which concludes the proof.
\end{proof}

The main application of Proposition~\ref{prop:rgcd_mat_zl} shall appear in the
following setting. Let $\mathcal O\subset \B$ be a maximal order, and
$I\subset \mathcal O$ be a left-ideal given by a $\Z$-basis $b_1, b_2, b_3,
b_4\in \mathcal O$. Assume that we can compute an isomorphism
$\phi:\mathcal O\otimes \Z_\ell\to \Mat_2(\Z_\ell)$. Then a generator of
$I\otimes \Z_\ell$ is $\phi^{-1}(\rgcd(\phi(b_1), \phi(b_2), \phi(b_3),
\phi(b_4)))$, so we can compute this generator by using
Proposition~\ref{prop:rgcd_mat_zl}.

\section{Computing $2$-dimensional isomorphisms}\label{sec:algo_isom}

In this section, which contains our main algorithms, we start by describing a family of automorphisms of product surfaces.
Then we give a criterion for an isomorphism $ (\varphi_{ij}): E_1\times E_2 \to E_1' \times E_2'$
to exist, if we fix two isogenies ($\varphi_{11}$ and $\varphi_{21}$) of its matrix representation. This criterion can be made effective; it will then be used to compute
(in polynomial time) isomorphisms between products of curves.

\subsection{The case of automorphisms}\label{sec:automorphisms}

\begin{prop}\label{prop:autofromiso}
  Let $E_1, E_2$ be two elliptic curves defined over $\overline{\F_p}$,
  $\varphi: E_1 \to E_2$ be an isogeny, and $a,b,c,d \in \Z$ be integers such that $ad - bc\deg(\varphi)  = \pm 1$.
  Then the endomorphism $\Phi=\begin{pmatrix}
    a & b\widehat{\varphi} \\
    c\varphi & d \\
  \end{pmatrix}\in\End(E_1\times E_2)$ 
  is an automorphism.
\end{prop}
\begin{proof}
  Direct computations show that the inverse of $\Phi$ is $\Phi^{-1}= \begin{pmatrix}
    d & -b\widehat{\varphi} \\
    -c\varphi & a \\
  \end{pmatrix}$.
  We could also compute $\deg(\Phi)=(ad - bc\deg(\varphi))^2 = 1$ by Proposition~\ref{prop:KaniDeg}.
  % $$\begin{array}{rcl}
  %   \deg(F) &=& (a^2 + c^2\deg(\varphi))(b^2\deg(\varphi) + d^2) - \deg( b\varphi a + d c \varphi) \\
  %           &=& (a^2 + c^2\deg(\varphi))(b^2\deg(\varphi) + d^2) - (ba+dc)^2\deg(\varphi) \\
  %           &=& a^2d^2 + c^2b^2\deg(\varphi)^2 - 2abcd\deg(\varphi) \\
  %           &=& (ad - bc\deg(\varphi))^2\\
  %           &=& 1.
  % \end{array}$$
\end{proof}

Proposition~\ref{prop:autofromiso} implies that if we are able to compute an isogeny $\varphi: E_1 \to E_2$,
then the $2$-dimensional morphism
$\begin{pmatrix}
  1+\deg(\varphi) & \widehat{\varphi} \\
  \varphi & 1 \\
\end{pmatrix}$
is an isomorphism.
Thus knowing the endomorphism rings of $E_1$ and $E_2$
is enough to compute automorphisms of the surface $E_1 \times E_2$,
thanks to~\cite[Algo.~5]{WesoFocs21}.

\subsection{Completion of matrices of isogenies}
In this section, we investigate the following question: given two isogenies
$\varphi_{11}, \varphi_{21}$, can we compute isogenies $\varphi_{12}, \varphi_{22}$ such that the matrix
$(\varphi_{ij})$ is an isomorphism. We give a necessary and sufficient
criterion for the existence of such isogenies $\varphi_{12}, \varphi_{22}$, and we provide an algorithm to compute them.
First we state a useful lemma.

\begin{lemm}\label{lemm:SumIdealCoprime}
  Let $F, E, E_1, E_2$ be elliptic curves over $\overline{ \F_p}$, $\psi:F \to E$
  be a (possibly inseparable) isogeny, and $\varphi_{1}:E \to E_1$, $\varphi_{2}:E \to E_2$ be
  separable isogenies of coprime degrees. Write $K:= \ker(\varphi_{1}) \oplus \ker(\varphi_{2})$. 
  Then
  \begin{equation*}
    \deg(\varphi_2) \Hom(E_1, F)\varphi_1 \psi  +  \deg(\varphi_1) \Hom(E_2,
    F)\varphi_2 \psi  = \Hom(E/K, F) \pi_K \psi.
  \end{equation*}
  where $\pi_K: E \to E/K$ is the canonical separable isogeny with kernel $K$.
\end{lemm}

\begin{proof}
  Set $d_1:=\deg(\varphi_1)$, $d_2:=\deg(\varphi_2)$, $I_1 :=d_2 \Hom(E_1, F)\varphi_1 \psi$,\\
  $I_2 := d_1 \Hom(E_2, F)\varphi_2 \psi$ and $I_K := \Hom(E/K, F)\pi_K \psi$.
  Direct computations show that
  $$\begin{array}{clcl}
    &\ker(d_2\varphi_1 \psi) \cap \ker(d_1 \varphi_2 \psi) & =& \psi^{-1}( \ker( d_2 \varphi_1) )
                                                            \cap
                                                            \psi^{-1}(\ker( d_1 \varphi_2)) \\
                                                          =& \psi^{-1} \left(  \ker( d_2 \varphi_1)
                                                            \cap \ker( d_2 \varphi_1) \right)
                                                          &=& \psi^{-1}\left(
                                                            ( \ker \varphi_1 + E[d_2] ) \cap
                                                            ( \ker \varphi_2 + E[d_1]) \right) \\
                                                           =& \psi^{-1}( \ker \varphi_1 \oplus \ker \varphi_2)
                                                          &=& \psi^{-1}( \ker( \pi_K)) \\
                                                           =& \ker ( \pi_K \psi).&&
  \end{array}$$

  Noticing that $I_1 + I_2 = I_{ \ker(d_2\varphi_1 \psi) \cap \ker(d_1 \varphi_2
  \psi)}$ and $\Hom(E/K, F) \pi_K \psi = I_{\ker ( \pi_K \psi)}$ concludes the proof.
  \end{proof}

We are now ready to state a key result of the paper.

\begin{theo}\label{theo:isom_iff}
   Let $E_1, E_2, E_1', E_2'$ be four isogenous elliptic curves defined over
   $\overline{ \F_p}$, and
   $\varphi_{11}:E_1\to E_1'$, $\varphi_{21}: E_1\to E_2'$
   be separable isogenies
   with coprime degrees.  There
   exist isogenies $\varphi_{12}:E_2\to E_1'$, $\varphi_{22}: E_2\to E_2'$
   such that $\phi=(\varphi_{ij})_{i,j\in\{1,2\}}\in \Hom(E_1\times E_2,
   E_1'\times E_2')$ is an isomorphism if and only if
   $ E_1 / \left( \ker(\varphi_{11})\oplus\ker(\varphi_{21}) \right)$ and $E_2$ are isomorphic.
\end{theo}

\begin{proof}
  Let $\psi: E_1\to E_2$ be an isogeny.
  Let $K$ denote the subgroup $\ker(\varphi_{11})\oplus\ker(\varphi_{21})$ of $E_1$. Let 
  $\pi_K:E_1 \to E_1/K$ be the associated canonical isogeny.
  Set $J_{11}:=\Hom(E_1', E_2)\varphi_{11}\widehat\psi$,
  $J_{21}:=\Hom(E_2',E_2)\varphi_{21}\widehat\psi$, and
  $J_K:=\Hom(E_1/K,E_2)\pi_K  \widehat \psi$,
  which are left-ideals in $\End(E_2)$.

  By Proposition~\ref{prop:isomorphism_suff_cond}, there
  exist isogenies $\varphi_{12}:E_2\to E_1'$, $\varphi_{22}: E_2\to E_2'$
  such that $\phi=(\varphi_{ij})_{i,j\in\{1,2\}}\in \Hom(E_1\times E_2,
  E_1'\times E_2')$ is an isomorphism if and only if there exist
  isogenies $\mu, \nu: E_1\to E_2$ which factors respectively through $\varphi_{11}$
  and $\varphi_{21}$ and such that $\deg(d_{21}\mu-d_{11}\nu) =
  d_{11}d_{21}$, where $d_{11}=\deg(\varphi_{11})$ and $d_{21} =
  \deg(\varphi_{21})$.
  Equivalently, $\deg((d_{21}\mu-d_{11}\nu) \widehat\psi) = d_{11}d_{21}\deg(\psi)$,
  with $\mu\widehat\psi\in J_{11}, \nu\widehat\psi\in J_{21}$.
  Since Lemma~\ref{lemm:SumIdealCoprime} implies that $J_K =
  d_{21}J_{11}+ d_{11}J_{21}$, it is equivalent to the existence of a $\sigma \in J_K$ such that
  $\deg( \sigma) = d_{11}d_{21} \deg(\psi)$. Remark that by definition, such a $\sigma \in J_K$
  would factor as $\sigma = \tau \pi_K \widehat \psi $ for some $\tau \in \Hom(E_1 / K , E_2)$.
  Thus the equation $\deg( \sigma )= d_{11}d_{21} \deg(\psi)$ is equivalent to
  $\deg( \tau) \deg( \pi_K) \deg(\psi)=d_{11}d_{21} \deg(\psi)$ by multiplicativity of the degree, which
  reduces to $\deg(\tau) = 1$, since $\deg(\pi_K) = d_{11}d_{21}$.

  Therefore, there
  exist isogenies $\varphi_{12}:E_2\to E_1'$, $\varphi_{22}: E_2\to E_2'$
  such that $\phi=(\varphi_{ij})\in \Hom(E_1\times E_2,
  E_1'\times E_2')$ is an isomorphism if and only if there exists $\tau \in \Hom(E_1 / K, E_2)$ with $\deg(\tau)=1$, i.e., $E_1 / K$ and $E_2$ are isomorphic.
\end{proof}

Theorem~\ref{theo:isom_iff} is actually effective, provided that we know the endomorphism rings
of the curves. Algorithm~\ref{algo:isomorphism_completion} computes such an
isomorphism.

\begin{prop}\label{prop:isomcompl_correct}
  Assuming GRH, Algorithm~\ref{algo:isomorphism_completion} is correct and it runs in time
  polynomial in $\log(p)$ and in the size of the input.
\end{prop}

\begin{proof}

  First we prove that Algorithm~\ref{algo:isomorphism_completion} is correct.
  In fact, Algorithm~\ref{algo:isomorphism_completion} follows the proof of
  Theorem~\ref{theo:isom_iff}. An isogeny associated to the ideal $I_\psi$
  plays the role of $\psi$ in the proof of Theorem~\ref{theo:isom_iff}. The
  ideals $I_{11}$, $I_{21}$ correspond to the isogenies $\varphi_{11},
  \varphi_{21}$ in Theorem~\ref{theo:isom_iff}, and the
  ideals $J_{11}$ and $J_{21}$ play the same role as in the proof of
  Theorem~\ref{theo:isom_iff}. We now prove that the endomorphism $\xi$ computed
  in Step~\ref{step:lattice_red} satisfies the requirements of $\sigma$ in the
  proof of Theorem~\ref{theo:isom_iff}, namely that $\Nrd(\xi) =
  d_{11}d_{21}\Nrd(I_{\psi})$. By the same argument as in the proof of
  Theorem~\ref{theo:isom_iff}, $J_K$ is a principal left-ideal (because $E_1/K \simeq E_2$) of reduced
  norm $d_{11}d_{21}\Nrd(I_{\psi})$, so it contains an element with this
  reduced norm; this proves that $\Nrd(\xi)= d_{11}d_{21}\Nrd(I_{\psi})$.
  Theorem~\ref{theo:isom_iff} also asserts that at least one of the
  matrices computed at Step~\ref{step:automorphisms} is an isomorphism.

  Let us now prove that the complexity is polynomial in the input
  size. Most steps reduce to $\Z$-linear algebra; this boils down to
  computing Hermite Normal Forms, which is
  done in time polynomial in
  the input size.
  Step~\ref{step:lattice_red} involves computing the shortest vector in a
  lattice of dimension $4$, with the positive definite quadratic
  form $(x_1, x_2, x_3, x_4)\mapsto x_1^2 + x_2^2 + p(x_3^2 + x_4^2)$. This is done in time polynomial in the
  input size and in $\log(p)$~\cite[Thm.~4.2.1]{nguyen2009low}.
  The combinatorial factor in Step~\ref{step:automorphisms} does not increase
  the complexity since the number
  of possible isogenies $E\to E'$ that are represented by the same left-ideal
  in $\End(E)$ equals the order of $\Aut(E)$. For most elliptic curves,
  $\Aut(E) = \{1, -1\}$, and in any case $\lvert\Aut(E)\rvert\leq
  24$~\cite[Appendix~A, Prop.~1.2.(c)]{Silv}.
  Assuming GRH, converting the ideal representation to an efficient representation can be
  done in polynomial-time, see e.g.~\cite[Appendix~C]{EffRpz}.
\end{proof}

\begin{algorithm}
  \caption{\IsomorphismCompletion{}\label{algo:isomorphism_completion}}
    \INPUT{Four maximal curves $E_1, E_1', E_2, E_2'$ over $\F_{p^2}$; $\Z$-bases of maximal orders $\mO[1],
     \mO[2]\subset \B$ and isomorphisms $\mO[1]\cong
    \End(E_1)$, $\mO[2]\cong\End(E_2)$; $\Z$-bases of left-ideals $I_{11}, I_{21}\subset\mO[1]$
    of coprime degrees which correspond to isogenies $\varphi_{11}:E_1\to E_1'$, $\varphi_{21}:E_1\to
    E_2'$ such that $E_2 \cong E_1 / \left( \ker(\varphi_{11}) \oplus \ker(\varphi_{21}) \right)$.
  }
    \OUTPUT{An efficient representation of a $2\times
    2$ matrix of isogenies ($\varphi_{ij}$) representing an isomorphism
    $E_1\times E_2\to
    E_1'\times E_2'$ such that $\Hom(E_2, E_1)\varphi_{11}\cong I_{11}$ and
    $\Hom(E_2', E_1)\varphi_{21}\cong I_{21}$.}
    Compute $d\in\Z$ such that $d\mO[1]\mO[2]\subset\mO[1]\cap \mO[2]$ and set
    $I_{\psi}:=d\mO[1]\mO[2]$, which is a connecting ideal between
    $\mO[1]$ and $\mO[2]$\tcp*{see \cite[Algo.~3.5]{kirschmer2010algorithmic}}
    
    Compute $\Z$-bases of $J_{11} := \overline I_\psi I_{11}$ and
    $J_{21} := \overline I_\psi I_{21}$\;
    Set $d_{11} := \Nrd(I_{11})$ and $d_{21} = \Nrd(I_{21})$\;
    Compute a $\Z$-basis of $J_K = d_{21}J_{11}+d_{11}J_{21}$\;
    
    Compute an element $\xi$ in $J_K$ whose reduced norm is
    minimal\;\label{step:lattice_red} \tcp{Lattice reduction in dimension
    $4$}
    
    Using linear algebra over $\Z$, compute $\xi_{11}\in J_{11},
    \xi_{21}\in J_{21}$ such that $d_{21}\xi_{11}-d_{11}\xi_{21} = \xi$\;
    
    Compute left-ideals $I_{12}$ and $I_{22}$ in the right-orders of
    $I_{11}$ and $I_{21}$ respectively, such that
    $\overline I_\psi I_{11} \overline I_{12} = \mO[2]\xi_{11}$ and $\overline I_\psi
    I_{21} \overline I_{22} =
    \mO[2]\xi_{21}$\;\tcp{Prop.~\ref{prop:complexityfactor} and
    Remark~\ref{remq:pb_red_facto}}
    Compute efficient representations of all possible matrices
    $(\varphi_{ij})$ such that $\varphi_{ij}\in\Hom(E_j, E_i')$ and $\Hom(E_i',
    E_j)\varphi_{ij}\cong I_{ij}$ as $\End(E_j)$
    left-modules\label{step:automorphisms}\;
    Using Proposition~\ref{prop:KaniDeg}, find a matrix among them which
    is an isomorphism and return it\;
\end{algorithm}

\subsubsection{Experiments}\label{para:expe1}

We present the first part of our experimental results. We would like to emphasize that the goal of our proof-of-concept implementation is to illustrate the practicality of the algorithms and to provide experimental evidence of their correctness. Our implementation is not optimized, and we do not claim that it is competitive in terms of speed and general efficiency.
The presented results, as well as those described in Paragraph~\ref{para:expe2},
are obtained runing our code under \texttt{Magma} version V2.28-3, using a 13th Gen Intel Core i7-1365U CPU.

Let us describe the file \textsf{ExperimentResults\_part1.mgm} available at~\cite{artifact}. 
We illustrate Algorithm~\ref{algo:isomorphism_completion}, 
\textit{i.e.}, we compute an isomorphism $E_0 \times E \to E_1' \times E_2'$, given isogenies $\varphi_{11}:E_0 \to E_1'$
and $\varphi_{21}:E_0 \to E_2'$ of coprime degree. Those isogenies are given by
left $\mO[0]$-ideals, denoted $I_{11}$ and $I_{21}$ respectively.

First we let the user choose the following parameters:
a lower bound for the prime $p$, and the norm of the two ideals $I_{11}$ and $I_{21}$.
Then we randomly compute such ideals, that determine (the isomorphism class of) $E$ according to
Theorem~\ref{theo:isom_iff}. We then recover the corresponding order $\mO \simeq \End(E)$ as the
right order of the left $\mO[0]$-ideal $I_K:=d_{11}I_{21} + d_{21}I_{11}$, by application of Lemma~\ref{lemm:SumIdealCoprime}.
We conclude by computing the ideals $I_{12}$ and $I_{22}$,
with a call to Algorithm~\ref{algo:isomorphism_completion}.
We can finally check that the ideals $I_{ij}$ represent an isomorphism, thanks to
the following lemma.

\begin{lemm}\label{lemm:CheckDeg}
   Let $E_1,E_2,E_1',E_2'$ be elliptic curves defined over $\overline{\F_p}$.
   For $i,j\in\{1,2\}$, let $\varphi_{ij}\in\Hom(E_j, E_i')$ be a morphism of
   degree $d_{ij}\in\N$. Denote $\Phi:E_1 \times E_2 \to E_1' \times E_2'$ the $2$-dimensional
   morphism given by the matrix $(\varphi_{ij})_{ij}$. Let $\psi:E_2 \to E_1$
   be a nonzero morphism.
   Then
   \begin{equation*}
     \deg \begin{pmatrix}
    \varphi_{11} \circ \psi & \varphi_{12} \\
    \varphi_{21} \circ \psi & \varphi_{22} \\           
 \end{pmatrix} = \deg(\psi) \deg(\Phi).
   \end{equation*}
 \end{lemm}

 \begin{proof}
   This is a direct consequence of Proposition~\ref{prop:KaniDeg}.
 \end{proof}

\begin{expl}
  We report on experimental results obtained by running the file
  \textsf{ExperimentResults\_part1.mgm} with seed $12345$.
  In this example, $p=503$ and $\ell_{11}=3$, $m_{11}=6$ and $\ell_{21}=5$, $m_{21}=4$. We denote by $1, i, j, k$ the usual generators
  of $\B$. 
  The input ideals are
  $I_{11}= \langle 729, \, 729\, i, 321/2 + 553\,\,  i + k/2, \, 553 + 1137 \, i/2 + j/2 \rangle$
  of norm $\ell_{11}^{m_{11}}$, and
  $I_{21} =   \langle 625, \, 625 \, i, \,  681/2 + 553\, i + k/2, \, 553 + 569\, i/2 + j/2 \rangle$
  of norm $\ell_{21}^{m_{21}}$.
  Then using Lemma~\ref{lemm:CheckDeg}, we check that the ideals $I_{11}, I_{21}$, and $I_{12} ,I_{22}$
  returned by \textsf{IsomorphismCompletion}, represent an isomorphism.

  With the same parameters $p, \ell_{11}, m_{11}, \ell_{21}, m_{21}$,
  and averaging 100 unseeded runs of \textsf{ExperimentResults\_part1.mgm}, we obtain a mean execution time of $0.157$ seconds.
  The memory usage remains constant at $32.09$MB.
\end{expl}
\subsection{Computing isomorphisms $E_0^2 \to E_1' \times E_0$}
\label{sc:fromE0}

In this section, we focus on a special $2$-dimensional instance of Problem~\ref{prob:effDOS}: we
assume that the endomorphism
rings of all curves are known, and that we also know subrings of $\End(E_1)$
and $\End(E_1')$ which are isomorphic to a low-discriminant imaginary quadratic
order. In this case, we provide a fast algorithm to compute the corresponding isomorphism,
described in Algorithm~\ref{algo:isom1}.
In order to simplify the exposition, we assume
that $E_1 = E_2 = E_2'$. In fact, this assumption does not lose any generality,
see Remark~\ref{remq:generalizationE0}. Also, for the sake of simplicity, we
assume throughout this section that the curve for which we know a subring of
endomorphisms isomorphic to a low-discriminant imaginary quadratic order is the
curve $E_0$ defined over $\F_{p^2}$ (with $p\equiv 3\bmod 4$) by the equation
$y^2 = x^3+x$. Its endomorphism ring contains a subring isomorphic to $\Z[i]$.
However, all the results presented in this section can be generalized without
any major difficulty to other curves. In particular, when $p\equiv 1\bmod 4$, a
curve with a low-discriminant can also be explicitly computed,
see~\cite[Sec.~3.1]{Deuring4People}.
In summary, our objective in this section is to provide a fast algorithm for
the following problem:
\begin{prob}[Low-discriminant Deligne-Ogus-Shioda problem]\label{prob:effDOSE0}
  Given a maximal supersingular elliptic curve $E_1'$ defined over
  $\mathbb F_{p^2}$ and its endomorphism ring, compute an
  $\F_{p^2}$-isomorphism $E_0\times E_0\to E_1'\times E_0$.
\end{prob}

The following statement gives sufficient conditions to use Theorem~\ref{theo:isom_iff}.

\begin{prop}\label{prop:SumKernelEndo}
   Let $E$, $E_1'$ be elliptic curves defined over
   $\overline{ \F_p}$ and let $\varphi:E\to E_1'$ be a separable isogeny. 
   Let $\alpha, \nu\in \End(E)$ be endomorphisms of coprime degrees such that $\deg(\alpha) =
   \deg(\varphi)$ and $\alpha\nu\in \Hom(E_1', E)\varphi$. Then
   $\ker(\nu)\oplus\ker(\varphi)$ is the kernel of the endomorphism 
   $\alpha\nu : E \to E $.
\end{prop}

\begin{proof}
  Since $\alpha\nu\in \Hom(E_1', E)\varphi$, we have
  $\ker(\varphi)\subset\ker(\alpha\nu)$. Consequently,
  $\ker(\varphi)+\ker(\nu)\subset \ker(\alpha\nu)$. The co-primality of
  $\deg(\nu)$ and $\deg(\varphi)$ implies that the intersection of the kernels is
  trivial. Since $\alpha$ and $\nu$ are separable, so is $\alpha\nu$ and
  hence $\lvert\ker(\alpha\nu)\rvert = \deg(\alpha)\deg(\nu) = \deg(\varphi)\deg(\nu) =
  \lvert\ker(\varphi)\oplus\ker(\nu)\rvert$, which shows that the inclusion is
  in fact an equality.
\end{proof}

Proposition~\ref{prop:SumKernelEndo} tells us that if are able to compute
$\varphi, \alpha$ and $\nu$, then Algorithm~\ref{algo:isomorphism_completion}
can compute an isomorphism $E \times E \to E_1' \times E$.
Our strategy will be to first fix $\varphi$, then to compute $\alpha$ and $\nu$ that satisfy the
conditions of Proposition~\ref{prop:SumKernelEndo}.
As explained above, we specialize to the case $E=E_0$, and $p \equiv 3 \bmod 4$.
The low-discriminant quadratic order will help us find the endomorphism $\alpha \in \End(E_0)$
of prescribed degree $\deg(\alpha)=\deg(\varphi)$ by solving low-discriminant
norm equations with Cornacchia's algorithm. Algorithm~\ref{algo:localgenerator}
provides a fast method for computing such $\alpha, \nu$ upon input of the ideal
$I$ corresponding to the isogeny $\varphi$.

We start with technical lemmas which state that computing $\alpha, \nu$
is related to computing a generator of a
localization of the ideal $I$ associated to $\varphi$ at a prime $\ell$.

\begin{lemm}\label{lem:technical_x_alpha1}
Let $\mathcal O\subset \B$ be a maximal order, $I\subset \mathcal O$ be a left
ideal, $\alpha \in \mathcal O$ such that $\Nrd(\alpha)=\Nrd(I)$, and $\ell\ne p$ be a prime number. Then
  the following statements are equivalent:
  \begin{itemize}
    \item[(a)] There exists $x\in \mathcal O$, such that $\alpha x\in I$ and $\Nrd(x)$ is
  not divisible by $\ell$;
\item[(b)] There exists an invertible $y\in \mathcal O\otimes
  \Z_{\ell}$ such that $\alpha y\in I\otimes\Z_{\ell}$.
  \end{itemize}
\end{lemm}

  \begin{proof}
First, we notice that all elements $x\in \mathcal O$ with reduced norm not divisible by
    $\ell$ are invertible in $\mathcal O\otimes \Z_\ell$; Indeed, $x\cdot
    (\overline x/\Nrd(x)) = 1$, hence the inverse of $x$ in $\mathcal
    O\otimes \Z_\ell$ is $\overline x/\Nrd(x)$. This proves the implication
    \emph{(a)}$\Rightarrow$\emph{(b)}.

    We now prove \emph{(b)}$\Rightarrow$\emph{(a)}. Let $y$ be as in
    \emph{(b)}. Let $b_1,\ldots, b_4$ be generators of $I$ seen as a free
    rank-$4$ $\Z$-module. Then $\alpha y=z_1\cdot b_1 + \cdots +
    z_4\cdot b_4$, with $z_1, \ldots, z_4\in \Z_\ell$. Next, pick
    integers
    $z_1',\ldots, z_4'\in \Z$ such that $z_i'\equiv
    z_i\bmod \ell^e$, where $e$ is the $\ell$-valuation of $\Nrd(\alpha)$. Then
    set $y'=\alpha^{-1}(z_1'\cdot b_1+\cdots + z_4'\cdot
    b_4)\in \B$.  We prove now
    that $x:=\Nrd(\alpha)y'/\ell^e$ satisfies the desired properties. First,
    we show that $x = \ell^{-e}\overline \alpha (z_1'\cdot b_1+\cdots + z_4'\cdot
    b_4)$ belongs to $\mathcal O$. Notice that $x$ clearly belongs to localized
    orders $\mathcal O \otimes \Z_{\ell'}$ for primes $\ell'\neq
    \ell$, so we only have to prove that $x\in\mathcal O\otimes \mathbb
    Z_\ell$. To do so, we use the fact that $z_i\equiv z_i'\bmod \ell^e$,
    hence there exists $z_1'',\ldots, z_4''\in\Z_\ell$ such that $z_i =
    z_i'+\ell^e z_i''$, which gives
    $$\begin{array}{rcl}x &=& \ell^{-e}\overline \alpha (z_1\cdot b_1+\cdots + z_4\cdot
    b_4) - \overline\alpha (z_1''\cdot b_1 +\cdots+z_4''\cdot b_4)\\
      &=& y \Nrd(\alpha)/\ell^e - \overline\alpha (z_1''\cdot b_1
      +\cdots+z_4''\cdot b_4),
      \end{array}$$
which shows that $x\in\mathcal O\otimes\Z_\ell$.
    Then we notice that
    $\Nrd(x)=(\Nrd(\alpha)/\ell^e)^2\Nrd(y')$ is not divisible by $\ell$ since
    $\Nrd(y')\equiv\Nrd(y)\not\equiv 0\bmod \ell$. Finally, since $\Nrd(\alpha)=\Nrd(I)\in I$,
    we notice that $\alpha x=\Nrd(\alpha)\ell^{-e}(z_1'\cdot b_1+\cdots + z_4'\cdot
    b_4)$ belongs to $I$
  \end{proof}

\begin{defiprop}\label{defiprop:ltype}
  Let $M\in \Mat_2(\Z_{\ell})$ be a matrix. We say that the
  $\ell$-type of $M$ is the pair of valuations (in $\N^2$) of the invariant
  factors of $M$, sorted in non-decreasing order. More
  explicitly, using the \emph{Smith Normal Form}, this means that $M$ has $\ell$-type $(d_1,d_2)$ if $d_1\leq d_2$
  and if there exist invertible matrices $S, T\in \GL_2(\Z_{\ell})$ such
  that 
  $$S\cdot M\cdot T = \begin{pmatrix}\ell^{d_1}&0\\0&\ell^{d_2}\end{pmatrix}.$$
 
    Since
  $\Mat_2(\Z_{\ell})$ is left-principal, we define the
  $\ell$-type of a left-ideal $I\subset \Mat_2(\Z_{\ell})$ as the $\ell$-type of a generator. 

  Let $I\subset \mathcal O$ be a left-ideal of a maximal order in $\B$, and
  $\ell\ne p$ be a prime number. By~\cite[Cor.~I.2.4]{QuatVigneras}, $\mathcal O\otimes \Z_{\ell}$ is isomorphic to $\Mat_2(\mathbb
  Z_{\ell})$, so we define the $\ell$-type of $I$ as the $\ell$-type of $I\otimes
  \Z_{\ell}$ regarded as an ideal in $\Mat_2(\Z_{\ell})$; this
  definition does not depend on the choice of the isomorphism $\mathcal
  O\otimes \Z_{\ell}\cong \Mat_2(\Z_{\ell})$.
  If $\alpha\in \mathcal O$ is an element in a maximal order, its $\ell$-type
  is defined as the $\ell$-type of the left-ideal $\mathcal O\alpha$.
\end{defiprop}

\begin{proof}
  The only thing that we need to prove is that the definition of the
  $\ell$-type of a left ideal $I\subset\mathcal
  O\otimes\Z_\ell\cong\Mat_2(\Z_\ell)$ does not depend on the choice of the
  isomorphism. In fact, this is a consequence of the fact that automorphisms of $\Mat_2(\mathbb
  Z_{\ell})$ preserve the $\ell$-type of matrices in $\Mat_2(\Z_\ell)$, which
  can be seen on the
  Smith Normal Form since automorphisms act as conjugations by invertible
  matrices.
\end{proof}

\begin{lemm}\label{lem:technical_x_alpha2}
  With the same notation as in Lemma~\ref{lem:technical_x_alpha1}, the
  $\ell$-types of $I$ and $\alpha$ are the same if and only if there exists an
  invertible $y\in \mathcal O\otimes \Z_{\ell}$ such that $\alpha y\in
  I\otimes\mathbb Z_{\ell}$.
\end{lemm}

\begin{proof}
  In this proof, we fix an isomorphism $\mathcal O\otimes \Z_\ell\cong
  \Mat_2(\Z_\ell)$ and we use it implicitly. Let $\beta\in
  \Mat_2(\Z_\ell)$ be a generator of $I\otimes \Z_\ell$. 

  To prove the ``only if'' part, we notice that if
  $\alpha$ and $\beta$ have the same invariant factors then there exist invertible
  matrices $U, V\in\GL_2(\Z_\ell)$ such that $U\cdot \beta= \alpha\cdot V$. 
  This implies that
  $\alpha\cdot V$ belongs to the left-ideal generated by $\beta$. Writing
  $y\in(\mathcal O\otimes\Z_\ell)^\times$ for the element corresponding to
  $V\in\GL_2(\Z_\ell)$, we obtain $\alpha y\in I\otimes \Z_\ell$.

  We now prove the ``if'' part of the statement. Under the fixed isomorphism, the
  assumption $\alpha y\in I\otimes\Z_\ell$ translates to the existence of
  $U\in\Mat_2(\Z_\ell)$ such that $\alpha y = U\beta$. By the multiplicativity
  of the determinant, we deduce that $U\in\GL_2(\Z_\ell)$. Therefore $\beta =
  U^{-1}\alpha y$ and hence $\beta$ and $\alpha$ have the same invariant
  factors.
\end{proof}

\begin{lemm}\label{lem:ltype_divl}
  With the same notation as in Lemma~\ref{lem:technical_x_alpha1}, if $\alpha\in\mathcal O$ has $\ell$-type $(i, j)$, then there exists
  $\alpha'\in\mathcal O$ with $\ell$-type $(0, j-i)$ such that $\alpha =
  \ell^i\,\alpha'$.
  Similarly, if $I\subset \mathcal O$ is a left-ideal which has $\ell$-type
  $(i,j)$, there exists a left-ideal $I'\subset\mathcal O$ with $\ell$-type
  $(0, j-i)$ and $I = \ell^i\cdot I'$.
\end{lemm}

\begin{proof}
  Using Definition-Proposition~\ref{defiprop:ltype}, we just need to prove this property for matrices
  in $\Mat_2(\Z_\ell)$. Let $M\in\Mat_2(\Z_\ell)$ be a matrix with $\ell$-type
  $(i,j)$, i.e., there exists $S,T\in\GL_2(\Z_\ell)$ such that 
  $$S\cdot M\cdot T = \begin{pmatrix}\ell^{i}&0\\0&\ell^{j}\end{pmatrix}.$$
    Then set $M' = S^{-1}\cdot
    \begin{pmatrix}1&0\\0&\ell^{j-i}\end{pmatrix}\cdot T^{-1}$, which has $\ell$-type $(0, j-i)$ and $M=\ell^i M'$.
\end{proof}

We are now ready to prove the main algorithmic result of this section.
It will rely on the fact that we can efficiently solve norm equations
in $\mO[0]$, if we assume that the norm to reach is big enough.

%To formalize
%that, we introduce the constant $c_0>0$ given by~\cite[Corollary 5.8]{WesoFocs21}.

\begin{algorithm}  
  \caption{\LocalGenerator{}\label{algo:localgenerator}}
  \INPUT{
    A prime $\ell$, a left ideal $I\subset \mathcal O_0$ not divisible by $\ell$,
    with reduced norm $\Nrd(I)=\ell^m$.
  }
  \OUTPUT{Returns elements $\alpha, x\in \mathcal O_0$ such that
    $\Nrd(\alpha)=\ell^m$ and $\alpha x$ generates $I\otimes \Z_\ell$.}
  Check that $m \log(\ell) > \log(p)^{c}$ and
  compute $\alpha \in \mO[0]$ such that $\Nrd( \alpha)=\ell^m$ \; \label{step:alpha}
  \tcp{Use \cite[Corollary 5.8]{WesoFocs21}; the constant $c$ is the one
    of this reference.}
  Using the isomorphism $\phi:\mathcal O_0\otimes\Z_\ell\to \Mat_2(\Z_\ell)$,
  compute $M_\alpha := \phi(\alpha_0 + \alpha_1 i +\alpha_2 j +\alpha_3 i
  j)\bmod \ell^m \in \Mat_2(\Z/\ell^m\Z)$\;
  Compute a generator $M_\beta \in \Mat_2(\Z/\ell^m\Z)$ of $\phi \left(I\otimes\Z/\ell^m\Z \right)$\;
  \tcp{This is done by computing the right-gcds of the
    four generators of $I\otimes \Z/\ell^m\Z\cong \Mat_2(\Z/\ell^m\Z)$,
    see Proposition~\ref{prop:rgcd_mat_zl}}
  Using the Smith Normal Form, compute two matrices $S, T\in\GL_2(\Z/\ell^m\Z)$ such that $S\cdot
  M_\alpha\cdot T = M_\beta$\;
  By lifting coordinates from $\Z/\ell^m\Z$ to representatives in
  $\Z$, set $x$ an element of $\mathcal O$ such that $x \equiv
  \phi^{-1}(T)\bmod \Z/\ell^m\Z$\;
  Return $(\alpha, x)\in \mathcal O_0^2$\;
\end{algorithm}

\begin{theo}\label{th:ComplextityAlgo1}

    Assuming GRH, there exists a constant $c>0$ such that on input an
    ideal $I$ with $\log(\Nrd(I)) > \log(p)^{c}$, 
    Algorithm~\ref{algo:localgenerator} is correct and terminates in
    expected polynomial time in $\log(p)$ and $\log(\Nrd(I))$.
  
\end{theo}

\begin{proof}

    Let $c$ be the constant given by~\cite[Corollary 5.8]{WesoFocs21}, so
    that the corresponding algorithm can be run on input $\Nrd(I)$.
    In~\cite{WesoFocs21}, the algorithm is proved to be correct and polynomial time
    under GRH. Therefore $\alpha$ can be correctly computed at
    Step~\ref{step:alpha}.
  Hence the correctness of Algorithm~\ref{algo:localgenerator} is a
  direct consequence of
  Lemma~\ref{lem:technical_x_alpha2} (to prove that the $\ell$-types of
  $I$ and $\alpha$ are the same) and of
  the proof of Lemma~\ref{lem:ltype_divl} which explains how to
  construct $x$.
  The complexity is a consequence of the polynomial time computation of $\alpha$, 
  and of the fact that computing the Hermite Normal Forms (for computing the generator of
  $I\otimes \Z/\ell^m\Z$) and Smith Normal Forms of matrices
  in $\Mat_2(\Z/\ell^m\Z)$ can be done in quasi-linear
  complexity $\widetilde
  O(m\log(\ell))$~\cite[Chap.~8]{storjohann2000algorithms}.
\end{proof}

Finally, we put all the pieces together and we give a complete algorithm
(Algorithm~\ref{algo:isom1}) to
compute an isomorphism $E_0^2\to E_1'\times E_0$ upon input of $E_1'$ and its
endomorphism ring. In the following, $c_0$ is the constant
in Theorem~\ref{th:ComplextityAlgo1}.

\begin{algorithm}  
  \caption{\textsc{LowDiscriminantIsomorphism}\label{algo:isom1}}
    \INPUT{A maximal elliptic curve $E_1'$ defined over
    $\F_{p^2}$, a maximal order $\mO[1]'$ together with an isomorphism
    $\mO[1]'\cong \End(E_1')$, a prime $\ell\ne p$.}
    \OUTPUT{An efficient representation of an isomorphism $E_0 ^2
    \to E_1' \times E_0$.}
    Compute a left $\mO[0]$-ideal $I_{11}$ with norm $\ell^m$ for
    some $m\in\N$ such that its right-order is conjugated to $\mO[1]'$
    and that $m\log(\ell) > \log(p)^{c_0}$\;
    \tcp{use KLPT~\cite{KLPT} or \cite[Algo.~5]{WesoFocs21}}
    Set $\alpha, \nu_1 := \LocalGenerator{}_{c_0}(\ell,I_{11})$\;
    Return \IsomorphismCompletion($E_0, E_1', E_0, E_0, \mO[0], \mO[1],
    I_{11}, \mO[0]\nu_1$)\;
\end{algorithm}

\begin{theo}\label{th:isom1}
  Algorithm~\ref{algo:isom1} is correct and, under GRH, it requires an expected number
  of operations which is polynomial in $\log(p)$.
\end{theo}

\begin{proof}
  The correctness is a consequence of
  Theorem~\ref{theo:isom_iff} and Proposition~\ref{prop:SumKernelEndo}, where Lemma~\ref{lem:technical_x_alpha1} ensures that the output of \LocalGenerator{} is valid data for the input of \IsomorphismCompletion{}.The correctness of the
  subroutines are given in Theorem~\ref{th:ComplextityAlgo1} and
  Proposition~\ref{prop:isomcompl_correct}. 
  The complexity follows from the complexities of the subroutines (Theorem~\ref{th:ComplextityAlgo1} and
  Proposition~\ref{prop:isomcompl_correct}), together with the fact that the
  output size of Wesolowski's algorithm~\cite[Algo.~5]{WesoFocs21} is
  polynomial in $\log(p)$.
\end{proof}

The following corollary shows that by running twice Algorithm~\ref{algo:isom1},
we can compute isomorphisms $E_1\times E_0\to E_2\times E_0$.

\begin{coro}\label{coro:thisom1}
  Let $E_1,E_2$ be maximal elliptic curves defined over $\F_{p^2}$, with
  known endomorphism rings $\mOi \cong \End(E_1)$ and
  $\mOd \cong \End(E_2)$. Assuming GRH we can compute an efficient representation of an $\F_{p^2}$-isomorphism from
  $E_1 \times E_0$ to $E_2 \times E_0$ with a Las Vegas probabilistic algorithm
  running in expected polynomial time.
\end{coro}

\begin{proof}
Running twice Algorithm~\ref{algo:isom1} provides us with efficient
  representations of two isomorphisms
  $\xi_1 : E_0^2 \to E_1\times E_0$ and $\xi_2: E_0^2\to E_2\times E_0$. By
  transposing $\xi_2$ as in Section~\ref{sec:transposing_isog}, we get an isomorphism $\xi_2': E_2\times E_0\to E_0^2$.
  Finally, computing a efficient representation of $\xi_1\circ \xi_2'$ provides
  the desired isomorphism.
\end{proof}

\begin{remq}\label{remq:generalizationE0}
  In fact, all results in this section generalize if we replace $E_0$ by an
  elliptic curve $E$ for which we know a subring of $\End(E)$ isomorphic to a
  low-discriminant imaginary quadratic order.
  Corollary~\ref{coro:thisom1} can actually be generalized as follows: given $E_1,
  E_2, E_1', E_2'$ supersingular elliptic curves defined over $\F_{p^2}$ with
  their endomorphism rings, and given low-discriminant endomorphisms in
  $\End(E_1)$ and $\End(E_1')$, we can efficiently compute an isomorphism $E_1\times E_2\to
  E_1'\times E_2'$ by computing isomorphisms $E_1\times E_2\to E_1^2$,
  $E_1^2\to E_1\times E_1'$, $E_1\times E_1'\to (E_1')^2$, and $(E_1')^2\to
  E_1'\times E_2'$. Each isomorphism can be computed by using the
  low-discriminant technique described in this section.
\end{remq}

\subsubsection{Experiments}\label{para:expe2}

Let us describe an implementation of instances of Algorithm~\ref{algo:isom1} we propose in the file
\textsf{ExperimentResults\_part2.mgm} available at~\cite{artifact}.
As in Section~\ref{para:expe1}, the
user can choose the parameters $p,\ell$ and $m$. Then we build a random ideal $I_{11}$
such that $\Nrd(I_{11})=\ell ^ m$. We recover $\alpha$ and $\nu_1$ with the
function $\textsf{LocalGenerator}$, where we compute $\alpha$ with to Cornacchia's algorithm.
Those computations allow us to set $I_{21}:= \mO[0] \nu_1$.
Finally, as described in Algorithm~\ref{algo:isom1}, we recover the two last ideals
$I_{12}$ and $I_{22}$ by a call to \textsf{IsomorphismCompletion}. The function
\textsf{RepresentIsomorphism} ensures that the computed ideals represent an
isomorphism.

\begin{expl}
  We report on experimental results obtained by running \texttt{Magma}
  on the file \textsf{ExperimentResults\_part2.mgm} with seed $12345$.
  In this example, $p=503$, $\ell=3$,
  $m=7$. We denote by $1, i, j, k$ the usual generators of $\B$.
  The input ideal $I_{11}$ has norm $\ell^m$ and is given by the basis:
  $ \langle 2187, 2187\, i, 2863/2 + 1551\, i + k/2, 1551 + 1511\, i/2 + j /2 \rangle$.
  Computations give
  $\alpha=30 + 28\, i - j$, and $\nu_1=136567 - 12039/2 \, i + 5127 \, j/2 - 2160 \, k$.
  We set $I_{21}:=\mO[0]\nu_1$. We finally check that \textsf{RepresentIsomorphism} returns \textsc{True} on the output of \textsf{IsomorphismCompletion}.

  With the same parameters $p, \ell$ and $m$,
  and averaging 100 unseeded runs of \textsf{ExperimentResults\_part2.mgm}, we obtain a mean execution time of $1.125$ seconds.
  The memory usage remains constant at $32.09$MB.
\end{expl}

\subsection{Computing an isomorphism in the general case}\label{sc:generalcase}

% In this section we conclude the proof of Theorem~\ref{theo:main2}.
% The authors would like to thank Benjamin Wesolowski, who quickly and clearly
% shows us that the tools developed above were enough to provide a
% complete proof of \ref{theo:main2}.

For simplicity, we focus on the case $p\equiv 3 \bmod 4$ where we can use the
curve
$E_0$ defined over $\F_{p^2}$ by the equation $y^2 = x^3 + x$ which admits
a low-discriminant endomorphism. If
$p\equiv 1\bmod 4$, we can use instead the construction described
in~\cite[Sec.~3.1]{Deuring4People}.

First we provide an algorithm to solve a special case of Theorem~\ref{theo:main2}.

\begin{algorithm}  
  \caption{\IsomorphismEO{}\label{algo:IsomorphismEO}}
    \INPUT{Two maximal elliptic curves $E_1,E_2$ defined over
    $\F_{p^2}$, maximal orders $\mO[1], \mO[2]$ together with isomorphisms
    $\mO[1]\cong \End(E_1), \, \mO[2]\cong \End(E_2)$}
    \OUTPUT{An efficient representation of an isomorphism $E_0^2
    \to E_1 \times E_2$.}
  Compute isogenies $\varphi_1: E_0 \to E_1$ and $\varphi_2: E_0 \to E_2$ 
  of coprime degrees, and their ideals $I_{1},I_{2}$;
  \tcp{use KLPT~\cite{KLPT} or \cite[Algo.~5]{WesoFocs21}}
  Set $I_K:=\deg(\varphi_1)I_2 + \deg(\varphi_2)I_1$, and compute its right order $\mO[3]$
along with an elliptic curve $E_3$ such that $\End(E_3) \cong \mO[3]$ \;
  \tcp{Lemma~\ref{lemm:SumIdealCoprime} ensures $E_3 \simeq E_0/(\ker(\varphi_1) \oplus \ker(\varphi_2))$}

  Compute an isomorphism $F: E_0 \times E_3 \to E_1 \times E_2$ by running
  \IsomorphismCompletion($E_0,E_3,E_1,E_2,\mO[0],\mO[3],I_1,I_2$) \;

  Choose a prime $\ell \neq p$, and compute an isomorphism $G:E_0^2 \to E_0 \times E_3$ by running
  \textsc{LowDiscriminantIsomorphism}($E_3,\mO[3],\ell$), and post-composing by the isomorphism $(P,Q) \mapsto (Q,P)$ \;

  Return the composition $F \circ G: E_0^2 \to E_1 \times E_2$\;
\end{algorithm}

\begin{algorithm}
  \caption{\ProductSurfacesIsomorphism{}}\label{algo:ProductSurfacesIsomorphism}
    \INPUT{Four maximal curves $E_1, E_1', E_2, E_2'$ over $\F_{p^2}$; $\Z$-bases of maximal orders $\mO[i],
     \mO[i]' \subset \B$ such that $\mO[i] \simeq \End(E_i)$ and $\mO[i]' \simeq \End(E_i')$ for $i \in \{1, 2 \}$.
  }
    \OUTPUT{An efficient representation of a $2\times
    2$ matrix of isogenies ($\varphi_{ij}$) representing an isomorphism
    $E_1\times E_2\to
    E_1'\times E_2'$}
  Compute two isomorphisms $\Phi : E_0^2  \to E_1 \times E_2$ and $\Psi: E_0^2 \to E_1' \times E_2'$, by running respectively
  $\IsomorphismEO(E_1,E_2,\mO[1],\mO[2])$ and $\IsomorphismEO(E_1',E_2',\mO[1]',\mO[2]')$ \;

  Compute an effective representation of an isomorphism $\widetilde{\Phi}: E_1 \times E_2 \to E_0^2$ by Corollary~\ref{coro:pseudodual}.

  Return an effective representation of the composition $\Psi \circ \widetilde{\Phi}$.
\end{algorithm}

\begin{theo}\label{theo:IsomE0}
  Algorithm~\ref{algo:IsomorphismEO} is correct and assuming GRH it requires an expected number of operations which
  is polynomial in $\log(p)$.
\end{theo}

\begin{proof}
  First we prove that the algorithm is correct. Since we know the endomorphism rings of $E_0,E_1,E_2$ we
  can compute isogenies $\varphi_i:E_0 \to E_i$ of coprime degrees by~\cite[Algo.~5]{WesoFocs21}.
  Applying Lemma~\ref{lemm:SumIdealCoprime} with $\psi = Id_{E_0}$ leads to the construction of the
  ideal $I_K$ representing an isogeny $E_0 \to E_0/( \ker(\varphi_1) \oplus \ker(\varphi_2))$.
  The correctness is then a consequence of Proposition~\ref{prop:isomcompl_correct} and Theorem~\ref{th:isom1}.
  The complexity follows from the same results, using the fact that
\cite[Algo.~5]{WesoFocs21} requires polynomial time assuming GRH, and that
constructing $I_K$ is standard $\Z$-linear algebra.  \end{proof}

%\begin{coro}\label{coro:IsomorphismEO}
%    Given $E_1', E_2'$, two maximal elliptic curves over
%    $\F_{p^2}$, with their endomorphism rings and assuming GRH, we can compute an $\F_{p^2}$-isomorphism
%    $ E_0^2 \to E_1' \times E_2'$ in polynomial time in $\log(p)$ and the size of the input.
%\end{coro}
%

We finally present a proof of Theorem~\ref{theo:main2}.

\begin{proof}[Proof of Theorem~\ref{theo:main2}]
By two calls to Algorithm~\ref{algo:IsomorphismEO}, we compute two isomorphisms
$\Phi: E_0^2  \to E_1 \times E_2  $ and $\Psi: E_0^2  \to E_1' \times E_2'$.
Then by Corollary~\ref{coro:pseudodual} we compute
the tranposed isomorphism $\widetilde{\Phi}:E_1 \times E_2  \to E_0^2  $.
Computing the composition $\Psi \circ \widetilde{\Phi}$ leads to the desired
isomorphism $E_1 \times E_2 \to E_1' \times E_2'$. The complexity statement follows from
Theorem~\ref{theo:IsomE0}.
\end{proof}

\bibliographystyle{abbrv}
\bibliography{biblio.bib}

\end{document}